\documentclass[11pt]{amsart}

\usepackage{a4wide}
\usepackage[T1]{fontenc}
\usepackage{amsmath}
\usepackage{amsfonts}
\usepackage{amssymb}
\usepackage{amsthm}
\usepackage{esint}
\usepackage{graphicx}
\usepackage{enumitem}
\usepackage{mathrsfs}
\usepackage{cite}
\usepackage{wasysym}
\usepackage{xr}
\usepackage[footnotesize]{caption}
\usepackage{color}
\usepackage{amsbsy}

\newtheorem{thm}{Theorem}[section]
\newtheorem*{thm*}{Theorem}
\newtheorem{prop}[thm]{Proposition}
\newtheorem{lem}[thm]{Lemma}
\newtheorem{cor}[thm]{Corollary}

\newtheorem{introthm}{Theorem}
\newtheorem{introprop}{Proposition}

\theoremstyle{definition}
\newtheorem{defin}[thm]{Definition}
\newtheorem{introdefin}{Definition}
\newtheorem{rem}[thm]{Remark}
\newtheorem{ex}[thm]{Example}

\DeclareMathOperator{\dist}{dist}
\DeclareMathOperator{\diam}{diam}
\DeclareMathOperator{\opspan}{span}
\DeclareMathOperator{\opint}{int}
\DeclareMathOperator{\oplk}{lk_2}
\DeclareMathOperator{\opang}{\sphericalangle}
\DeclareMathOperator{\opid}{id}
\DeclareMathOperator{\aff}{aff}
\DeclareMathOperator{\simp}{\bigtriangleup}

\DeclareMathOperator{\opim}{im}

\DeclareMathOperator{\BAP}{BAP_m}

\DeclareMathOperator{\graph}{Graph}

\newcommand{\A}{\mathcal{A}}

\newcommand{\E}{\mathcal{E}}
\newcommand{\M}{\mathcal{M}}
\newcommand{\G}{\mathscr{G}}
\newcommand{\Dxy}{\mathfrak{D}}
\newcommand{\bad}{\mathfrak{B}}
\newcommand{\CDxy}{\overline{\Dxy}}
\newcommand{\Dom}{\mathcal{D}}

\newcommand{\N}{\mathbb{N}}
\newcommand{\Z}{\mathbb{Z}}
\newcommand{\Sphere}{\mathbb{S}}
\newcommand{\Cone}{\mathbb{C}}
\newcommand{\Ball}{\mathbb{B}}
\newcommand{\CBall}{\overline{\mathbb{B}}}
\newcommand{\Disc}{\mathbb{D}}
\newcommand{\CDisc}{\overline{\mathbb{D}}}
\newcommand{\Shell}{\mathbb{A}}
\newcommand{\R}{\mathbb{R}}
\newcommand{\HM}{\mathcal{H}}
\newcommand{\HD}{d_{\mathcal{H}}}
\newcommand{\dgras}{d_{\mathrm{Gr}}}
\newcommand{\K}{\mathcal{K}}
\newcommand{\Kphi}{\K_{l,\varphi}}

\newcommand{\face}{\mathfrak{fc}}
\newcommand{\height}{\mathfrak{h}}
\newcommand{\hmin}{\height_{\min}}
\newcommand{\Slk}{\mathscr{S}}
\newcommand{\red}{\rho\varepsilon}

\newcommand{\ere}{\boldsymbol{\epsilon}_{\red}}
\newcommand{\laka}{\lambda\kappa}
\newcommand{\fin}{\mathrm{g}}

\newcommand{\proj}{\pi}
\newcommand{\pperp}{\pi^{\perp}}

\newcommand{\mat}[1]{\mathbf{#1}}

\newcommand{\nbeta}{\beta_m^{\Sigma}}
\newcommand{\ntheta}{\theta_m^{\Sigma}}

\newcommand{\ahl}{\mathrm{Ahl}}

\title[Sobolev embedding using Menger curvature]
{Geometric Sobolev-like embedding using high-dimensional Menger-like curvature}
\author{S{\l}awomir Kolasi{\'n}ski}
\address{Institute of~Mathematics\\
  University of~Warsaw\\
  Banacha~2, 02-097 Warsaw\\
  Poland}
\email{s.kolasinski@mimuw.edu.pl}
\urladdr{http://www.mimuw.edu.pl/~skola}
\date{\today}

\begin{document}

\begin{abstract}
  We study a~modified version of~Lerman-Whitehouse Menger-like curvature defined
  for $(m+2)$ points in~an $n$-dimensional Euclidean space. For $1 \le l \le
  m+2$ and an~$m$-dimensional set $\Sigma \subset R^n$ we also introduce global
  versions of~this discrete curvature, by taking supremum with respect to
  $(m+2-l)$ points on~$\Sigma$. We then define geometric curvature energies
  by~integrating one of~the~global Menger-like curvatures, raised to a~certain
  power~$p$, over all $l$-tuples of~points on~$\Sigma$. Next, we~prove that
  if~$\Sigma$ is compact and $m$-Ahlfors regular and if $p$ is greater than
  the~dimension of~the~set of~all $l$-tuples of~points on~$\Sigma$ (i.e.~$p >
  ml$), then the P.~Jones' $\beta$-numbers of~$\Sigma$ must decay as~$r^{\tau}$
  with $r \to 0$ for some $\tau \in (0,1)$. If~$\Sigma$ is an~immersed $C^1$
  manifold or a~bilipschitz image of~such set then it~follows that it~is
  Reifenberg flat with vanishing constant, hence (by~a~theorem of~David, Kenig
  and~Toro) an~embedded $C^{1,\tau}$ manifold. We~also define a~wide class of
  other~sets for which this assertion is~true. After~that, we~bootstrap the
  exponent $\tau$ to the~optimal one~$\alpha = 1 - ml/p$ showing an~analogue
  of~the~Morrey-Sobolev embedding theorem $W^{2,p} \subseteq
  C^{1,\alpha}$. Moreover, we obtain a~qualitative control over the~local graph
  representations of~$\Sigma$ only in~terms~of~the~energy.
\end{abstract}

\maketitle

\section*{Introduction}

Menger curvature is defined for three points $x_0$, $x_1$, $x_2$ in $\R^n$ as
follows
\begin{displaymath}
  \mathbf{c}(x_0,x_1,x_2) = \frac{4 \HM^2(\simp(x_0,x_1,x_2))}{|x_0 - x_1| |x_1 - x_2| |x_2 - x_0|} \,,
\end{displaymath}
where $\HM^l$ denotes the $l$-dimensional Hausdorff measure and
$\simp(x_0,\ldots,x_l)$ is the convex hull of the set
$\{x_0,\ldots,x_l\}$. Using the sine theorem one easily sees that
$\mathbf{c}(x_0,x_1,x_2)$ is just the inverse of~the radius of~the circumcircle
of~$\simp(x_0,x_1,x_2)$. Let $\gamma \subseteq \R^3$ be a~closed, Lipschitz
curve with arc-length parameterization~$\Gamma$, i.e. $\Gamma : S_L \to \R^3$ is
such that $\gamma = \Gamma(S_L)$ and $|\Gamma'| = 1$ a.e. - here $S_L = \R/L\Z$
denotes the circle of~length $L$. We set
\begin{displaymath}
  \mathbf{c}_0[\gamma] = \sup_{x_0,x_1,x_2 \in \gamma} \mathbf{c}(x_0,x_1,x_2) \,,
  \quad
  \mathbf{c}_1[\gamma](x_0) = \sup_{x_1,x_2 \in \gamma} \mathbf{c}(x_0,x_1,x_2) \,,
\end{displaymath}
\begin{displaymath}
  \mathbf{c}_2[\gamma](x_0,x_1) = \sup_{x_2 \in \gamma} \mathbf{c}(x_0,x_1,x_2)
  \quad \text{and} \quad
  \mathbf{c}_3[\gamma](x_0,x_1,x_2) = \mathbf{c}(x_0,x_1,x_2) \,.
\end{displaymath}
Using these quantities we define
\begin{displaymath}
  \Delta[\gamma] = \mathbf{c}_0[\gamma]^{-1}
  \quad \text{and for $i = 1,2,3$} \quad
  \M_p^i(\gamma) = \int_{(\gamma)^i} \mathbf{c}_i^p[\gamma]\ d\HM^i \,,
\end{displaymath}
where $(\gamma)^i$ is the Cartesian product of~$i$~copies of~$\gamma$. Gonzalez
and~Maddocks~\cite{MR1692638} suggested that these functionals can serve
as~\emph{knot energies}, i.e.~energies which separate knot types by infinite
energy barriers. Gonzalez, Maddocks, Schuricht and von
der~Mosel~\cite{MR1883599} showed that whenever $\mathbf{c}_0[\gamma] < \infty$
then $\gamma$ is an~embedded (without self-intersections) manifold of class
$C^{1,1} = W^{2,\infty}$. The functionals $\M_p^1$, $\M_p^2$ and $\M_p^3$ poses
a~similar property. For $i = 1,2,3$ if~$\M_p^i(\gamma) < \infty$ for some $p >
i$ then $\gamma$ is an~embedded manifold of~class $C^{1,1-i/p}$ (see the
articles by~Strzelecki, Szuma{\'n}ska and von
der~Mosel~\cite{MR2489022,MR2668877} and by~Strzelecki and~von
der~Mosel~\cite{MR2318572}). Furthermore, in~\cite{MR2318572} the authors proved
that $\M_p^1(\gamma)$ is finite \emph{if and only if} $\gamma$ is an image of
a~$W^{2,p}$ function. Later Blatt~\cite{blatt-note} showed that for $i = 2,3$
and $p > i$ the energy $\M_p^i(\gamma) < \infty$ \emph{if and only if} $\gamma$
belongs to the Sobolev-Slobodeckij space $W^{1+s,p}$, where $s = 1 -
\frac{i-1}p$. Note that, $W^{1+s,p}(\R) \subseteq C^{1,1-i/p}(\R)$ whenever $p >
i$, so these results deliver geometric counterparts of~the Sobolev-Morrey
embedding.

For $p$ below the critical level (i.e. $p < i$) one cannot expect that
finiteness of $\M_p^i(\gamma)$ implies smoothness. Scholtes~\cite{1202.0504}
showed that if $\gamma$ is a polygon in $\R^2$ then $\M_p^i(\gamma) < \infty$ if
and only~if $p < i$. For a~$1$-dimensional Borel set $E \subseteq \R^2$ a~famous
result of~David and~L{\'e}ger~\cite{MR1709304} says that $\M_2^3(E)$ is finite
if and only if $E$ is~rectifiable. This was a~crucial step in the proof
of~Vitushkin's conjecture characterizing removable sets $E$ for bounded analytic
functions.

There are some generalizations of these results to higher dimensions. Lerman
and~Whitehouse~\cite{MR2558685,0805.1425} suggested a~few possible definitions
of~discrete curvatures of Menger-type. They used these curvatures
to~characterize uniformly rectifiable measures in the sense of David
and~Semmes~\cite{MR1251061}. In~this article we use a~modified version (having
different scaling) of one of~the~quantities introduced in~\cite{MR2558685}.

Our research has been motivated directly by the work of Strzelecki and von
der~Mosel~\cite{0911.2095}, where the authors work with $2$-dimensional surfaces
in~$\R^3$. They define the discrete curvature of~four points $x_0,x_1,x_2,x_3 \in
\R^3$ by the formula
\begin{displaymath}
  \K_{SvdM}(x_0,x_1,x_2,x_3) =
  \frac{\HM^3(\simp(x_0,x_1,x_2,x_3))}{\HM^2(\partial \simp(x_0,x_1,x_2,x_3)) \diam(x_0,x_1,x_2,x_3)^2} \,.
\end{displaymath}
For $\Sigma \subseteq \R^3$ a~compact, closed, connected, Lipschitz surface they
also define
\begin{displaymath}
  \M^{SvdM}_p(\Sigma) = \int_{\Sigma} \int_{\Sigma} \int_{\Sigma} \int_{\Sigma}
  \K_{SvdM}(x_0,x_1,x_2,x_3)^p
  \ d\HM^2_{x_0}\ d\HM^2_{x_1}\ d\HM^2_{x_2}\ d\HM^2_{x_3} \,.
\end{displaymath}
In~\cite{0911.2095} the authors prove that if~$\M^{SvdM}_p(\Sigma) \le E <
\infty$ for some $p > 8 = \dim(\Sigma^4)$, then $\Sigma$ has~to~be an~embedded
manifold of~class~$C^{1,1-8/p}$ with local graph representations whose domain
size is controlled solely in terms if $E$ and~$p$. This additional control of
the graph representations allowed them to prove~\cite[Theorem~1.5]{0911.2095}
that any sequence $(\Sigma_j)_{j \in \N}$ of~compact, closed, connected,
Lipschitz surfaces containing the origin and with uniformly bounded measure and
energy, i.e. $\M^{SvdM}_p(\Sigma_j) \le E$ and $\HM^2(\Sigma_j) \le A$ for each
$j \in \N$, contains a~subsequence~$\Sigma_{j_l}$, which converges in $C^1$
topology to some $C^{1,1-8/p}$ compact, closed, connected manifold. This in turn
allowed them to solve some variational problems with topological constraints
(see~\cite[Theorems~1.6 and~1.7]{0911.2095}).

Similar regularity results were also obtained by Strzelecki and von
der~Mosel~\cite{1102.3642} for yet another energy
\begin{displaymath}
  \E_p^{tp}(\Sigma) = \int_{\Sigma} \int_{\Sigma} R_{tp}(x,y)^{-p}\ d\HM^m_x\ d\HM^m_y \,,
  \quad \text{where} \quad
  R_{tp}(x,y) = \frac{|x-y|^2}{2\dist(y-x,T_x\Sigma)}
\end{displaymath}
and $T_x\Sigma$ is the tangent space to $\Sigma$ at $x$. The quantity
$R_{tp}(x,y)$ is called \emph{the tangent-point radius}, because it measures
the~radius of the sphere tangent to~$\Sigma$ at~$x$ and passing through~$y$.  If
$\Sigma$ is a~closed, connected, Lipschitz surface with $\E_p^{tp}(\Sigma) <
\infty$ for some $p > 2m$, then $\Sigma \in C^{1,1-(2m)/p}$.

In this paper we define energy functionals for $m$-dimensional subsets $\Sigma$
of~$\R^n$ (we always assume $m \le n$) and we study regularity of sets with
finite energy. For $m+2$ points $x_0,\ldots,x_{m+1}$ in $\R^n$ we set
(cf. \cite[\S6.1.1]{MR2558685})
\begin{displaymath}
  \K(x_0, \ldots, x_{m+1}) = \frac{\HM^{m+1}(\simp(x_0, \ldots, x_{m+1}))}
  {\diam(x_0, \ldots, x_{m+1})^{m+2}}
\end{displaymath}
and for $p > 0$ and $l = 1,2,\ldots,m+2$ we define\footnote{If $l = m+2$ there
  are $m+2$ integrals and no supremum.}
\begin{displaymath}
  \E_p^l(\Sigma) = \int_{\Sigma^l} \sup_{x_l, \ldots x_{m+1} \in \Sigma}
  \K(x_0,\ldots, x_{m+1})^p\ d \HM^{ml}_{x_0, \ldots, x_{l-1}} \,.
\end{displaymath}
We prove that these functionals can be called \emph{geometric curvature
  energies}, i.e.~for~sets $\Sigma$ of~relatively little smoothness, finiteness
of~the~energy guarantees both embeddedness and higher regularity.

Of~course, the~condition $\E_p^l(\Sigma) < \infty$ cannot guarantee
that~$\Sigma$ is a~manifold (even for large~$p$) just for any $m$-dimensional
set~$\Sigma$. The main issue is that $\E_p^l(\Sigma \setminus A) \le
\E_p^l(\Sigma)$ for any set $A$, so creating holes in $\Sigma$ decreases
the~energy. Hence, we need to work with a restricted class of~sets. We~say that
$\Sigma$ is \emph{locally lower Ahlfors regular}~if
\begin{equation}
  \label{def:ahlfors}
  \tag{\texttt{Ahl}}
  \exists R_{\ahl} > 0 \ 
  \exists A_{\ahl} > 0 \ 
  \forall x \in \Sigma \ 
  \forall r \le R_{\ahl}
  \quad
  \HM^m(\Sigma \cap \Ball(x,r)) \ge A_{\ahl} r^m \,.
\end{equation}
Here $\Ball(x,r)$ denotes the $n$-dimensional open ball of~radius~$r$ centered
at~$x$. We~also need a~variant of~the P.~Jones' \emph{beta} numbers introduced
in~\cite{MR1103619} and the~\emph{bilateral beta} numbers, which originated from
Reifenberg's work~\cite{MR0114145} and his famous topological disc theorem
(see~\cite{Sim96} for a~modern proof). We define
\begin{displaymath}
  \nbeta(x,r) = \frac 1r \inf_{H \in G(n,m)} \sup_{z \in \Sigma \cap \CBall(x,r)} \dist(z,x+H)
\end{displaymath}
\begin{displaymath}
  \text{and} \quad
  \ntheta(x,r) = \frac 1r \inf_{H \in G(n,m)} \HD(\Sigma \cap \CBall(x,r), (x + H) \cap \CBall(x,r)) \,,
\end{displaymath}
\begin{displaymath}
  \text{where} \quad
  \HD(E,F) = \sup_{y \in E} \dist(y,F) + \sup_{y \in F} \dist(y,E)
\end{displaymath}
is the Hausdorff distance and $G(n,m)$ denotes the Grassmannian
of~$m$-dimensional linear subspaces of~$\R^n$. The $\beta$-number measures
the~flatness of~$\Sigma$ in~a~given scale in a~scaling invariant way.
The~$\theta$-number measures additionally the size of holes in~that scale.
Using these notions we can formulate our first
\begin{introprop}
  \label{prop:beta-est}
  Let $\Sigma \subseteq \R^n$ be a~compact set satisfying \eqref{def:ahlfors}
  and let $l \in \{1, \ldots, m+2\}$. If~$\E_p^l(\Sigma) \le E < \infty$ for
  some $p > ml$, then there exists a constant $C_A = C_A(m,l,p)$ such that
  \begin{displaymath}
    \forall r \le R_{\ahl} \ 
    \forall x \in \Sigma
    \quad
    \nbeta(x,r) \le C_A \left( \frac{E}{A_{\ahl}^l} \right)^{\frac 1\kappa} r^{\frac{\lambda}{\kappa}} \,,
  \end{displaymath}
  where $\kappa = (p+ml)(m+1)$ and~$\lambda = p - ml$.
\end{introprop}
Applying the result of David, Kenig and~Toro~\cite[Proposition~9.1]{MR1808649}
(cf. Proposition~\ref{prop:dkt-reg}) we then obtain
\begin{introthm}
  \label{thm:regularity}
  Let $\Sigma \subseteq \R^n$ be a~compact set satisfying \eqref{def:ahlfors}
  and such that
  \begin{equation}
    \tag{$\theta\lesssim\beta$}
    \label{def:theta-beta}
    \exists R_{\theta\beta} > 0 \ 
    \exists M_{\theta\beta} > 1 \ 
    \forall x \in \Sigma \ 
    \forall r \le R_{\theta\beta}
    \quad
    \ntheta(x,r) \le M_{\theta\beta} \nbeta(x,r) \,.
  \end{equation}
  If~$\E_p^l(\Sigma) < \infty$ for some $p > ml$, then $\Sigma$ is a~closed,
  embedded manifold of~class $C^{1,\lambda/\kappa}$.
\end{introthm}
This motivates the following
\begin{introdefin}
  \label{def:fine-sets}
  We say that a~set $\Sigma \subseteq \R^n$ is an~\emph{$m$-fine set} if it~is
  $m$-dimensional, compact and~satisfies \eqref{def:ahlfors}
  and~\eqref{def:theta-beta}.
\end{introdefin}

Examples of $m$-fine sets include closed $m$-dimensional Lipschitz submanifolds
of~$\R^n$ and also images of~maps $\varphi : M \to \R^n$, where $M$ is
an~abstract, closed $C^1$~manifold and $\varphi$ is an~immersion. Other examples
are~described in~Section~\ref{sec:fine-sets}.

The condition~\eqref{def:theta-beta} is purely geometric but it is hard to
understand what kind of behavior it implies. It gives control over the size
of~holes in~$\Sigma$ but it does not imply that the~topological boundary
of~$\Sigma$ is~empty. In~\cite[Definition~2.9]{1102.3642}
(cf. Definition~\ref{def:adm}) the authors considered a~class
of~\emph{admissible sets} satisfying a~different set of~conditions. Their idea
was to use the topological linking number to~prevent holes in~$\Sigma$.
Any~admissible set in the~sense of~\cite{1102.3642} with finite $\E_p^l$-energy
for some $p > ml$, satisfies the~\eqref{def:theta-beta} condition
(see~\cite[Theorem~4.15]{slawek-phd} for the case $l = m+2$), hence, by
Theorem~\ref{thm:regularity}, it is a~closed $C^{1,\lambda/\kappa}$-manifold.

Once we have estimates on the $\beta$-numbers (Proposition~\ref{prop:beta-est}),
the~regularity result (Theorem~\ref{thm:regularity}) follows quite easily but
the~key point is that one can get a~uniform (not depending on $\Sigma$) control
over the local graph representations of~$\Sigma$ only in terms of the energy
bound $E$ and the parameters $m$, $l$ and~$p$. To show that this is true
we~first prove the~following uniform, with respect to $\Sigma$, estimate
on~the~local lower Ahlfors regularity of~$\Sigma$.
\begin{introthm}
  \label{thm:uahlreg}
  Let $\Sigma \subseteq \R^n$ be an $m$-fine set. If $\E_p^l(\Sigma) \le E <
  \infty$ for some $p > ml$, then
  \begin{displaymath}
    \exists R_0 = R_0(E,m,l,p) > 0 \ 
    \forall x \in \Sigma \ 
    \forall r \le R_0
    \quad
    \HM^m(\Sigma \cap \Ball(x,r)) \ge \left( \tfrac{\sqrt{15}}{4} \right)^m \omega_m r^m \,,
  \end{displaymath}
  where $\omega_m = \HM^m(\Ball(0,1) \cap \R^m)$ is the measure of the unit ball
  in $\R^m$.
\end{introthm}

Theorem~\ref{thm:regularity} together with Theorem~\ref{thm:uahlreg} give
us~estimates on the $\beta$-numbers independent of~$\Sigma$. Knowing that
$\Sigma$ is a~compact, closed, $C^{1,\lambda/\kappa}$-submanifold of~$\R^n$,
we~prove that also the constant $M_{\theta\beta}$ from
the~\eqref{def:theta-beta} condition can be replaced by an~absolute
constant. Then we~obtain estimates on~the~oscillation of~tangent planes
of~$\Sigma$ solely in~terms~of $E$, $m$, $l$ and~$p$. This allows to prove that
the~size of~a~single patch of~$\Sigma$ representable as~a~graph of~some function
is controlled solely in~terms~of $E$, $m$, $l$ and~$p$. Next we bootstrap the
exponent $\frac{\lambda}{\kappa}$ to the~optimal one $\alpha = 1 - \frac{ml}p$
(see~\cite{SKMS} and~\cite{blatt-kola} for the~proof that this is~indeed
optimal).
\begin{introthm}
  \label{thm:optimal-reg}
  Let $\Sigma \subseteq \R^n$ be an $m$~fine set. If $\E_p^l(\Sigma) \le E <
  \infty$ for some $p > ml$, then $\Sigma$ is a~closed
  $C^{1,\alpha}$-manifold. Moreover, there exist two constants $R_{\fin} =
  R_{\fin}(E,m,l,p) > 0$ and~$C_{\fin} = C_{\fin}(E,m,l,p) > 0$ such that
  \begin{displaymath}
    \forall x \in \Sigma \ 
    \exists F_x \in C^{1,\alpha}(T_x\Sigma,(T_x\Sigma)^{\perp}) \ 
    \quad
    \Sigma \cap \Ball(x,R_{\fin}) = \graph(F_x) \cap \Ball(x,R_{\fin})
  \end{displaymath}
  \begin{displaymath}
    \text{and} \quad
    \forall y,z \in T_x\Sigma
    \quad
    \| DF_x(y) - DF_x(z) \| \le C_{\fin} |y-z|^{\alpha} \,,
  \end{displaymath}
  where $\graph(F_x) = \{ z \in \R^n : \exists y \in T_x\Sigma\ \ z = y + F_x(y) \}$.
\end{introthm}

This work already lead to a~few other results. In our~joint work with
Szuma{\'n}ska~\cite{SKMS} we~have constructed an~example of~a~function $f \in
C^{1,\alpha_0}([0,1]^m)$, where $\alpha_0 = 1 - \frac {m(m+1)}p$, whose graph
has infinite $\E_p^{m+2}$-energy and we~proved that for any $\alpha_1 >
\alpha_0$ the graphs of~$C^{1,\alpha_1}$ functions always have finite
energy. Later this result was complemented by our joint work with
Blatt~\cite{blatt-kola}, where we have shown that a~$C^1$-submanifold of~$\R^n$
has finite $\E_p^l$-energy for some $p > m(l-1)$ and $l \in \{ 2, \ldots, m+2
\}$ \emph{if and only if} it is locally a~graph of~a~function
in~the~Sobolev-Slobodeckij space $W^{1+s,p}$, where $s = 1 -
\frac{m(l-1)}p$. In~another article~\cite{w2pchar} written jointly with
Strzelecki and von der~Mosel, we have shown that an~$m$-fine set $\Sigma
\subseteq \R^n$ is~a~$W^{2,p}$-manifold \emph{if and only if} it satisfies
the~condition $\E_p^1(\Sigma) < \infty$. The paper \cite{w2pchar} includes
Theorem~\ref{thm:optimal-reg} for the $\E_p^1$-energy and a~counterpart
of~Theorem~\ref{thm:optimal-reg} for a~modified version
of~the~$\E_p^{tp}$-energy, where one integration was replaced by taking the
supremum. In~a~forthcoming joint article with Strzelecki and~von
der~Mosel~\cite{W2pcharII} we also prove a~compactness result similar
to~\cite[Theorem~1.5]{0911.2095} for the $\E_p^l$ and~$\E_p^{tp}$ energies.

\subsection*{Organization of the paper}
In Section~\ref{sec:prelim} we describe the notation, we state precisely the
result of~\cite{MR1808649} about Reifenberg flat sets with vanishing constant
and we prove some auxiliary propositions about roughly regular simplices and
about the metric on the Grassmannian. In~\ref{sec:p-energy} we also show that
$C^2$-manifolds have finite $\E_p^l$-energy for any $p > 0$.
In~Section~\ref{sec:first-reg} we prove Proposition~\ref{prop:beta-est} and
Theorem~\ref{thm:regularity} and we give some examples of $m$-fine sets.
In~Section~\ref{sec:adm-sets} we establish Theorem~\ref{thm:uahlreg}. For this
we~need to define another class of admissible sets and prove some more auxiliary
results about cones and homotopies inside cones. In
Section~\ref{sec:tangent-planes} we prove a counterpart
of~Theorem~\ref{thm:optimal-reg}, where $\alpha$ is replaced with
$\lambda/\kappa$. In~Section~\ref{sec:improved-holder} we bootstrap the exponent
$\lambda/\kappa$ to the optimal $\alpha = 1 - \frac{ml}p$ and consequently
establish Theorem~\ref{thm:optimal-reg}.



\section{Preliminaries}
\label{sec:prelim}

\subsection{Notation}
We write $\Sphere$ for the unit $(n-1)$-dimensional sphere centered at~the
origin and we write $\Ball$ for the unit $n$-dimensional open ball centered
at~the origin. We also use the symbols $\Sphere_r = r \Sphere$, $\Ball_r = r
\Ball$, $\Sphere(x,r) = x + r \Sphere$ and~$\Ball(x,r) = x + r \Ball$.

If $v = (v_1,\ldots,v_n)$ is a~vector in~$\R^n$, we write $|v| = \sqrt{\sum
  |v_i|^2} = \sqrt{\langle v, v \rangle}$ for the standard Euclidean norm
of~$v$. If $\mat A~: \R^k \to \R^l$ is a~linear operator, we write $\| \mat A \|
= \sup_{|v| = 1} |\mat Av|$ for the operator norm of~$\mat A$.

The symbol $G(n,m)$ denotes the Grassmann manifold of~$m$-dimensional linear
subspaces of~$\R^n$. Whenever we write $U \in G(n,m)$ we identify the point $U$
of~the space $G(n,m)$ with the appropriate $m$-dimensional subspace of~$\R^n$.
In~particular any vector $u \in U$ is treated as an $n$-dimensional vector
in~the ambient space $\R^n$ which happens to lie in~$U \subseteq \R^n$.

If $A$ is any set, then we write $\opid_A : A \to A$ for the identity mapping.
Let $H \in G(n,m)$. We~use the symbol $\proj_H$ to denote the orthogonal
projection onto $H$ and $\pperp_H = I - \proj_H$ to denote the orthogonal
projection onto the orthogonal complement $H^{\perp}$. We write $\aff\{x_0,
\ldots, x_m \}$ for the smallest affine subspace of~$\R^n$ containing points
$x_0, \ldots, x_m \in \R^n$, i.e.
\begin{displaymath}
  \aff\{x_0, \ldots, x_m \} = x_0 + \opspan\{x_1-x_0, \ldots, x_m-x_0\} \,.
\end{displaymath}
Let $\mathbf{T} = \simp(x_0, \ldots, x_k)$. We set
\begin{itemize}
\item $\face_i \mathbf{T} = \simp(x_0, \ldots, \widehat{x_i}, \ldots, x_k)$ - the $i$-th face of~$\mathbf{T}$,
\item $\height_i(\mathbf{T}) = \dist(x_i, \aff\{ x_0, \ldots, \widehat{x_i} , \ldots, x_k \}$ - the height lowered from $x_i$,
\item $\hmin(\mathbf{T}) = \min\{ \height_i(\mathbf{T}) : i = 0,1,\ldots,k \}$ - the minimal height of~$\mathbf{T}$.
\end{itemize}
In~the course of~the proofs we will frequently use cones and ''conical caps''
of~different sorts. We define
\begin{itemize}
\item $\Cone(\delta,H) = \{ x \in \R^n : |\pperp_H(x)| \ge \delta |x| \}$ - the
  \emph{cone} with ''axis'' $H^{\perp}$ and ''angle'' $\delta$,
\item $\Shell(r,R) = \Ball_R \setminus \overline{\Ball}_r$ - the open
  \emph{shell} (or the \emph{$n$-annulus}) of~radii $r$ and $R$,
\item $\Cone(\delta,H,r,R) = \Cone(\delta,H) \cap \Shell(r,R)$ - the
  \emph{conical cap} with ''angle'' $\delta$, ''axis'' $H^{\perp}$ and radii~$r$
  and~$R$ as the intersection of~a~cone with a~shell.
\end{itemize}

\begin{rem}
  We use the notation $C = C(x,y,z)$ to denote that $C$ depends \emph{solely}
  on~$x$, $y$ and~$z$. The symbols $C$, $\hat{C}$, $\tilde{C}$, $\bar{C}$ are
  used to denote general constants, whose values may change in different parts
  of the text. Subscripts in constants (like ``$C_{\theta}$'') \emph{do~not
    denote dependences} but are used to name the constant and distinguish it
  from other constants. Subscripted constants always have global meaning and do
  not change.
\end{rem}

\subsection{Reifenberg flat sets}

For convenience we introduce the following
\begin{defin}
  \label{def:beta-plane}
  Let $\Sigma \subseteq \R^n$ be any set. Let $x \in \Sigma$ and $r > 0$. We
  say that $H \in G(n,m)$ is the \emph{best approximating $m$-plane} for
  $\Sigma$ in~$\CBall(x,r)$ and write $H \in \BAP(x,r)$ if the
  following condition is satisfied
  \begin{displaymath}
    \HD(\Sigma \cap \CBall(x,r), (x + H) \cap \CBall(x,r))
    \le \ntheta(x,r) \,.
  \end{displaymath}
\end{defin}
Since $G(n,m)$ is compact, such $H$ always exists, but it might not be unique,
e.g.~consider the set $\Sigma = \Sphere \cup \{ 0 \}$ and take $x=0$, $r=2$.

Recall the definitions of $\nbeta$ and $\ntheta$ given in the introduction.
In~\cite{MR1808649}, the authors define the~$\beta$ and $\theta$ numbers in
a~slightly different way using open balls instead of~closed ones. This does~not
change much since both definitions lead to comparable quantities
(see~\cite[Proposition~1.35]{slawek-phd})

\begin{defin}[cf.~\cite{MR1808649}, Definition~1.3]
  \label{def:rfvc}
  We say that a~closed set $\Sigma \subseteq \R^n$ is \emph{Reifenberg-flat with
    vanishing constant} (of~dimension $m$) if for every compact subset $K
  \subseteq \Sigma$
  \begin{displaymath}
    \lim_{r \to 0} \sup_{x \in K} \ntheta(x,r) = 0 \,.
  \end{displaymath}
\end{defin}

The following proposition was proved by David, Kenig and Toro.
\begin{prop}[cf.~\cite{MR1808649}, Proposition~9.1]
  \label{prop:dkt-reg}
  Let $\tau \in (0,1)$ be given. Suppose $\Sigma$ is a~Reifenberg-flat set with
  vanishing constant of~dimension $m$ in~$\R^n$ and that, for each compact subset
  $K \subseteq \Sigma$ there is a~constant $C_K$ such that
  \begin{displaymath}
    \nbeta(x,r) \le C_K r^{\tau} \quad \text{for each $x \in K$ and $r \le 1$.}
  \end{displaymath}
  Then $\Sigma$ is a~$C^{1,\tau}$-submanifold of~$\R^n$.
\end{prop}

\subsection{Voluminous simplices}

Here we define the~class of~\emph{$(\eta,d)$-voluminous} simplices, where $\eta$
measures the~``regularity'' of~a~simplex. The curvature $\K$ of~any such simplex
is~controlled in~terms~of $\eta$ and~$d$. A~very similar notion was used
by~Lerman and~Whitehouse in~\cite[\S~3.1]{MR2558685}, where these kind~of
simplices were called $1$-separated. We~derive estimates of~the~distance
by~which we~can move each vertex of~an~$(\eta,d)$-voluminous simplex without
losing the~lower bound on~the~curvature. We will use this result to~obtain
a~lower bound on the $\E_p^l$-energy in~the~proof
of~Proposition~\ref{prop:eta-d-balance}.

\begin{defin}
  Let $\mathbf{T} = \simp(x_0,\ldots,x_k)$ be a~simplex in~$\R^n$ and let $d \in
  (0,\infty)$ and $\eta \in (0,1)$. We say that $\mathbf{T}$ is
  \emph{$(\eta,d)$-voluminous} if
  \begin{displaymath}
    \diam(\mathbf{T}) \le d
    \quad \text{and} \quad
    \hmin(\mathbf{T}) \ge \eta d \,.
  \end{displaymath}
\end{defin}

\begin{rem}
  \label{rem:reg-curv}
  If $\mathbf{T} = \simp T$ is $(\eta,d)$-voluminous then
  \begin{displaymath}
    \frac{(\eta d)^k}{k!}  \le \HM^k(\mathbf{T}) \le \frac{d^k}{k!} \,,
    \quad \text{hence} \quad
    \K(T) \ge \frac{\eta^k}{k! d}  \,.
  \end{displaymath}
\end{rem}

Let us recall the~definition of~the~outer product:
\begin{defin}
  \label{def:wedge}
  Let $w_1$, \ldots, $w_l$ be vectors in $\R^n$. We define \emph{the outer
    product $w_1 \wedge \cdots \wedge w_l$} to be the~vector in~$\R^{\binom
    nl}$, whose coordinates are exactly the~$l$-minors of~the~$(n \times
  l)$-matrix $(w_1,\ldots,w_l)$.
\end{defin}

\begin{rem}
  \label{rem:wedge-meas}
  A~standard fact from linear algebra says that the length $|w_1 \wedge \cdots
  \wedge w_l|$ of~the~outer product of~$w_1$, \ldots, $w_l$ is~equal
  to~the~$l$-dimensional volume of the parallelotope spanned by $w_1$, \ldots,
  $w_l$. In~particular $|w_1 \wedge \cdots \wedge w_l| \le |w_1| \cdot |w_2|
  \cdots |w_k|$.
\end{rem}

\begin{prop}
  \label{prop:vol-close}
  Let $\mathbf{T}_0 = \simp T_0 = \simp(x_0,\ldots,x_k)$ be an
  $(\eta,d)$-voluminous simplex in~$\R^n$. There exists a~number $\varsigma_k =
  \varsigma_k(\eta) \in (0,1)$ such that for~any~simplex $\mathbf{T}_1 = \simp
  T_1 = \simp(y_0,\ldots,y_k)$ satisfying $|x_i - y_i| \le \varsigma_k d$
  for~each $i = 1, \ldots, k$ the~following estimate
  \begin{equation}
    \label{est:vol-close}
    \frac 34 \HM^k(\mathbf{T}_0) \le \HM^k(\mathbf{T}_1) \le \frac 54 \HM^k(\mathbf{T}_0)
    \quad \text{holds, hence also} \quad
    \K(T_1) \ge \frac{3 \eta^k}{4 k! d} \,.
  \end{equation}
\end{prop}

\begin{proof}
  Let $\tilde{\varsigma} \in (0,1)$ be some number and~let~$T_1 =
  (y_0,\ldots,y_k)$ be~such that $|x_i - y_i| \le \tilde{\varsigma} d$ for each
  $i=1,\ldots,k$. We set $v_i = x_i - x_0$ and $w_i = (y_i - y_0) - v_i$, where
  $i = 1,\ldots,k$.
  \begin{align*}
    \HM^k(\mathbf{T}_1) &= \frac 1{k!} |(v_1 + w_1) \wedge \ldots \wedge (v_k + w_k)| \\
    &= \frac 1{k!} |(v_1 \wedge \ldots \wedge v_k) 
    + (w_1 \wedge v_2 \wedge \ldots \wedge v_k) + (v_1 \wedge w_2 \wedge \ldots \wedge v_k) \ldots \\
    &\phantom{= \frac 1{k!} |} \ldots + (w_1 \wedge w_2 \wedge v_3 \wedge \ldots \wedge v_k) + \ldots
    + (w_1 \wedge w_2 \wedge w_3 \wedge \ldots \wedge w_k) | \,.
  \end{align*}
  Whenever we take an outer product of $j$ vectors from the set~$\{ w_1, \ldots,
  w_k \}$ and~$(k-j)$ vectors from the set~$\{ v_1,\ldots,v_k \}$ we obtain
  a~vector of length at~most $d^{k-j} (\tilde{\varsigma} d)^j$. Hence we~can
  write
  \begin{displaymath}
    |(w_1 \wedge v_2 \wedge \ldots \wedge v_k) + \ldots + (w_1 \wedge w_2 \wedge \ldots \wedge w_k)|
    \le \sum_{j=1}^k \binom kj d^k \tilde{\varsigma}^j = d^k((1 + \tilde{\varsigma})^k - 1)\,,
  \end{displaymath}  
  \begin{displaymath}
    \text{which gives}\quad
    \HM^k(\mathbf{T}_0) - d^k((1 + \tilde{\varsigma})^k - 1) 
    \le \HM^k(\mathbf{T}_1) 
    \le \HM^k(\mathbf{T}_0) + d^k((1 + \tilde{\varsigma})^k - 1) \,.
  \end{displaymath}
  Since $\mathbf{T}_0$ is $(\eta,d)$-voluminous, it satisfies
  $\HM^k(\mathbf{T}_0) \ge \frac 1{k!}(\eta d)^k$. We~set
  \begin{equation}
    \label{def:varsigma}
    \varsigma_k = \left(1+\frac{\eta^k}{4k!}\right)^{\frac 1k} -  1 \,,
  \end{equation}
  so that $d^k((1 + \tilde{\varsigma})^k - 1) \le \frac 14 \HM^k(\mathbf{T}_0)$.
  Thus, if $|x_i - y_i| \le \varsigma_k d$, then we obtain the desired estimate
  $\frac 34 \HM^k(\mathbf{T}_0) \le \HM^k(\mathbf{T}_1) \le \frac 54
  \HM^k(\mathbf{T}_0)$.
\end{proof}

\begin{rem}
  Let $x,s \in \R$ and $s > 0$. When $|x| \approx 0$, the~function $(1+x)^s$
  behaves asymptotically like $1 + sx$, hence there exists a~constant
  $C_{\varsigma} = C_{\varsigma}(k) > 1$ such that
  \begin{equation}
    \label{eq:sigma-asymp}
    \forall \eta \in (0,1) \quad
    \frac{1}{C_{\varsigma}} \eta^k \le \varsigma_k(\eta) \le C_{\varsigma} \eta^k \le \tfrac 14\,.
  \end{equation}
\end{rem}

\subsection{The $\E_p^l$-energy for smooth manifolds}
\label{sec:p-energy}

Observe that $\K(\alpha T) = \frac 1{\alpha}\K(T)$ for any $\alpha > 0$, so our
curvature behaves under scaling like the original Menger curvature $\mathbf{c}$.
If~$\simp T$ is~a~regular simplex (meaning that all the side lengths are equal),
then $\K(T) \simeq \frac{1}{\diam T} \simeq R(T)^{-1}$, where $R(T)$ is the
radius of~a~circumsphere of~$T$. For $m = 1$ one easily sees that we always have
$\K(T) \le c(T) = R^{-1}(T)$. In dimension $m=2$ we also have $\K(T) \le 4\pi
\K_{SvdM}(T)$ for any $T$ and $\K(T) \simeq \K_{SvdM}(T)$ if $T$ is~a~regular
simplex.

We emphasis the behavior on regular simplices because small, close to regular
(or \emph{voluminous}) simplices are the reason why $\E_p^l(\Sigma)$ might get
very big or infinite. For the class of~$(\eta,d)$-voluminous simplices $T$
the~value $\K(T)$ is comparable with yet another possible definition of~discrete
curvature (cf. \cite[\S 10]{0805.1425})
\begin{displaymath}
  \K'(T) = \frac{\hmin(\simp T)}{\diam(T)^2} 
  = \frac{1}{\diam(T)} \frac{\hmin(\simp T)}{\diam(T)} \,,
\end{displaymath}
which is basically $\frac 1{\diam(T)}$ multiplied by a~scale-invariant
''regularity coefficient'' $\frac{\hmin(\simp T)}{\diam(T)}$. This last factor
prevents $\K'$ from blowing up on simplices with vertices on smooth manifolds.

It occurs that one cannot define $k$-dimensional Menger curvature using
integrals of~$R^{-1}$. This ''obvious'' generalization of~the~Menger curvature
fails because of~examples (see~\cite[Appendix B]{0911.2095}) of~very smooth
embedded manifolds for which this kind of~curvature would be~unbounded.
For the~curvature~$\K$ we have the following
\begin{prop}
  \label{prop:C2mani}
  If $M \subseteq \R^n$ is a~compact, $m$-dimensional, $C^2$ manifold embedded
  in~$\R^n$ then the discrete curvature $\K$ is bounded on $M^{m+2}$. Therefore
  $\E_p^l(M)$ is finite for every $p > 0$ and every $l \in \{ 1, \ldots, m+2
  \}$.
\end{prop}

\begin{lem}
  \label{lem:beta-curv}
  Let $\Sigma \subseteq \R^n$ be any set and let $T = (x_0,\ldots,x_{m+1}) \in
  \Sigma^{m+2}$. We set $\mathbf{T} = \simp T$ and $d = \diam(\mathbf{T})$.
  There exists a constant $C_{\K\beta} = C_{\K\beta}(m,n)$ such that we have
  \begin{displaymath}
    \HM^{m+1}(\mathbf{T}) \le C_{\K\beta} \nbeta(x_0,d) d^{m+1} 
    \quad \text{and consequently} \quad
    \K(T) \le C_{\K\beta} \frac{\nbeta(x_0,d)}{d} \,.
  \end{displaymath}
\end{lem}

\begin{proof}
  \footnote{The author wishes to thank Simon Blatt for significantly simplifying
    this proof while we were working on~\cite{blatt-kola}.}  Without loss of
  generality we can assume that $x_0 = 0$. If the vectors $\{x_1, \ldots,
  x_{m+1}\}$ are not linearly independent, then $\HM^{m+1}(\mathbf{T}) = 0$ and
  there is nothing to prove.

  Let $x_1, \ldots x_{m+1}$ be linearly independent and let $W$ denote the
  $(m+1)$-dimensional vector space spanned be these vectors. Set
  \begin{equation*}
    \mathbf{S} = \{ s \in W^{\perp} : |s| \le \nbeta(x_0,d)d \} \,.
  \end{equation*}
  Then, the set $\mathbf{T} + \mathbf{S}$ is isometric with $\mathbf{T} \times
  \mathbf{S}$ and the following holds
  \begin{equation}
    \label{eq:VolumeST}
    \HM^{n}(\mathbf{T} + \mathbf{S}) 
    = \HM^{m+1}(\mathbf{T}) \HM^{n-m-1}(\mathbf{S})
    = \omega_{n-m-1} \HM^{m+1}(\mathbf{T}) d^{n-m-1} \nbeta(0,d)^{n-m-1} \,.
  \end{equation}

  Using compactness of the Grassmannian we can find a~vector space $V \in
  G(n,m)$ such that
  \begin{displaymath}
    \sup_{y \in \Sigma \cap \CBall(x_0,d)} |\pperp_{V_j}(y)| = \nbeta(x_0,d) d \,.
  \end{displaymath}
  Observe also that the mapping $Q : G(n,m) \to \R^n$ given by $Q(V) = P_V(y)$
  is continuous for any choice of $y \in \R^n$. In~consequence, we get the
  estimate
  \begin{displaymath}
    \forall y \in \Sigma \cap \Ball(x_0,d)
    \quad
    |\pperp_V(y)| \le \nbeta(x_0,d) d  \,.
  \end{displaymath}
  The vertices of $\mathbf{T}$ lie in $\Sigma \cap \CBall(x_0,d)$ and
  $\mathbf{T}$ is convex, so we also have
  \begin{displaymath}
    \forall t \in \mathbf{T}
    \quad
    |\pperp_V(t)| \le \nbeta(x_0,d) d \,.
  \end{displaymath}
  Let $y \in \mathbf{T}+\mathbf{S}$ and let $t \in \mathbf{T}$ and $s \in
  \mathbf{S}$ be such that $s+t = y$. Using the triangle inequality we see that
  \begin{displaymath}
    |\proj_V(y)| \le |y| \le (1 + \nbeta(0,d)) d  
  \end{displaymath}
  \begin{displaymath}
    \text{and} \quad
    |\pperp_V(y)| \le |\pperp_V(t)|  + |\pperp_V(s)| \le 2 \nbeta(x_0,d) d \,.
  \end{displaymath}
  Hence, $\mathbf{T}+\mathbf{S}$ is a subset of
  \begin{displaymath}
    Z = \big\{ y \in \R^n : |\proj_V (y)| \le 2 d,\, |\pperp_V(y)| \le 2 \nbeta(0,d) d \big\} \,.
  \end{displaymath}
  and we obtain
  \begin{equation}
    \label{eq:EstimateST}
    \HM^n(\mathbf{T}+\mathbf{S}) \le \HM^n(Z) = \omega_m \omega_{n-m} 2^n \nbeta(0,d)^{n-m} d^n \,.
  \end{equation}
  Combining \eqref{eq:VolumeST} and \eqref{eq:EstimateST} we obtain the desired
  estimate.
\end{proof}

\begin{cor}
  \label{cor:eta-beta}
  Let $\Sigma \subseteq \R^n$ be any set and let $T = (x_0,\ldots,x_{m+1}) \in
  \Sigma^{m+2}$. There exists a~constant $C_{\eta\beta} = C_{\eta\beta}(n,m)$
  such that if $\simp T$ is $(\eta,d)$-voluminous then the parameters $\eta$ and
  $d$ must satisfy
  \begin{displaymath}
    \eta \le C_{\eta\beta} \nbeta(x_0,d)^{\frac 1{m+1}} \,.
  \end{displaymath}
\end{cor}

\begin{proof}
  Recalling Remark~\ref{rem:reg-curv} we have the estimate $\HM^{m+1}(\simp T)
  \ge ((m+1)!)^{-1}(\eta d)^{m+1}$, which, combined with
  Lemma~\ref{lem:beta-curv}, leads to $\eta \le ((m+1)! C_{\K\beta})^{\frac
    1{m+1}} \nbeta(x_0,d)^{\frac 1{m+1}}$.
\end{proof}

\begin{proof}[Proof of~Proposition~\ref{prop:C2mani}]
  Since $M$ is a~compact $C^2$-manifold, it has a~tubular neighborhood
  \begin{displaymath}
    M_{\varepsilon} = M + \overline{B}_{\varepsilon}
    = \{ x + y : x \in M,\, y \in \overline{B}_{\varepsilon} \}
  \end{displaymath}
  of~some radius $\varepsilon > 0$ and the nearest point projection $p :
  M_{\varepsilon} \to M$ is a~well-defined, continuous function (see
  e.g.~\cite{MR0110078} for a~discussion of~the properties of~the nearest point
  projection mapping). To find $\varepsilon$ one proceeds as follows. Take the
  principal curvatures $\kappa_1,\ldots,\kappa_m$ of~$M$. These are continuous
  functions $M \to \R$, because $M$ is a~$C^2$ manifold. Next set
  \begin{displaymath}
    \varepsilon = \sup_{x \in M} \max \{ |\kappa_1|, \ldots, |\kappa_m| \} \,.
  \end{displaymath}
  Such maximal value exists due to continuity of~$\kappa_j$ for each $j =
  1,\ldots,m$ and compactness of~$M$.
  
  We will show that for all $r \le \varepsilon$ and all $x \in \Sigma$ we have
  \begin{equation}
    \label{eq:C2-beta}
    \nbeta(x,r) \le \frac 1{2 \varepsilon} r \,.
  \end{equation}
  Next, we apply Lemma~\ref{lem:beta-curv} and get the desired result.

  Choose $r \in (0,\varepsilon]$. Fix some point $x \in \Sigma$ and pick a~point
  $y \in T_xM^{\perp}$ with $|x - y| = \varepsilon$. Note that $y$ belongs to
  the tubular neighborhood $M_{\varepsilon}$ and that $p(y) = x$. Hence, the
  point $x$ is the only point of~$M$ in~the ball
  $\CBall(y,\varepsilon)$. In~other words $M$ lies in~the complement
  of~$\CBall(y,\varepsilon)$. This is true for any $y$ satisfying $y \in
  T_xM^{\perp}$ and $|x - y| = \varepsilon$, so we have
  \begin{displaymath}
    M \subseteq \R^n \setminus \bigcup \left\{
      \CBall(y,\varepsilon) :
      y \perp T_{x}M,\, |y - x| = \varepsilon
    \right\} \,.
  \end{displaymath}
  Pick another point $\bar{x} \in \Sigma \cap \CBall(x,r)$. We then have
  \begin{equation}
    \label{eq:x-pos}
    \bar{x} \in \CBall(x,r) \setminus \bigcup
    \left\{
      \CBall(y,\varepsilon) :
      y \perp T_{x}M,\, |y - x| = \varepsilon
    \right\} \,.
  \end{equation}
  \begin{figure}[!htb]
    \centering
    \includegraphics{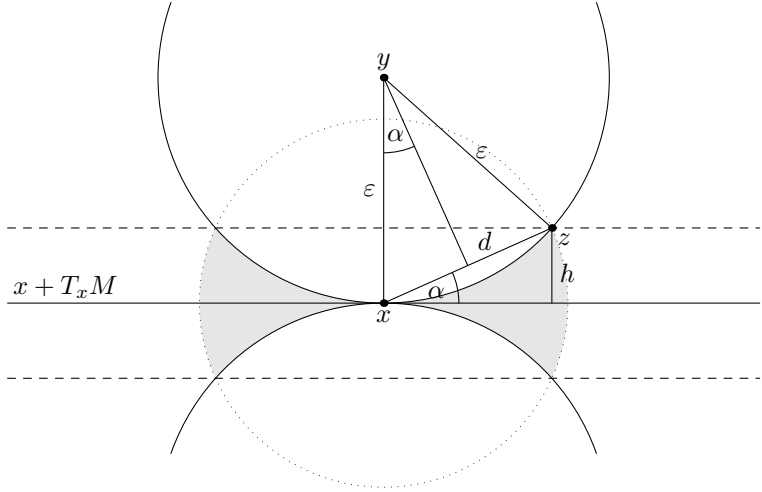}
    \caption{All of~$M \cap \CBall(x,r)$ lies in~the grey area. The point
      $\bar{x}$ lies in~the complement of~$\Ball(y,\varepsilon)$ and inside
      $\CBall(x,r)$ so it has to be closer to $T_{x}M$ than $z$.}
    \label{F:height}
  \end{figure}
  Using \eqref{eq:x-pos} and simple trigonometry, it is ease to calculate the
  maximal distance of~$\bar{x}$ from the tangent space $T_xM$. Let $z$ be any
  point in~the intersection $\partial \Ball(x,r) \cap \partial
  \Ball(y,\varepsilon)$. Note that points of~$M \cap \CBall(x,\varepsilon)$ must
  be closer to $T_xM$ than $z$. In~other words
  \begin{equation}
    \label{eq:z-max-dist}
    \forall x \in M \cap \Ball(x,r)
    \quad
    \dist(x,T_xM) \le \dist(z,T_xM) \,.
  \end{equation}
  This situation is presented on Figure \ref{F:height}. Let $\alpha$ be the
  angle between $T_xM$ and $z$ and set $h = \dist(z,T_xM)$. We use the fact
  that the distance $|z - x|$ is equal to $r$.
  \begin{equation} \label{eq:h-calc}
    \sin \alpha = \frac{|z - x|}{2 \varepsilon} = \frac{h}{|z - x|}
    \quad \Rightarrow \quad
    h = \frac{|z - x|^2}{2 \varepsilon} = \frac{r^2}{2 \varepsilon} \,.
  \end{equation}
  This shows \eqref{eq:C2-beta} and thus finishes the proof.
\end{proof}

\begin{rem}
  Note that the only property of~$M$, which allowed us to prove
  Proposition~\ref{prop:C2mani} was the existence of~an appropriate tubular
  neighborhood $M_{\varepsilon}$. One can easily see that
  Proposition~\ref{prop:C2mani} still holds if $M$ is just a~set
  of~\emph{positive reach} as defined in~\cite{MR0110078}.
\end{rem}

\subsection{The metric on the Grassmannian}
Recall that formally, $G(n,m)$ is defined as the homogeneous space
\begin{displaymath}
  G(n,m) = O(n) / (O(m) \times O(n-m)) \,,
\end{displaymath}
where $O(n)$ is the orthogonal group; see e.g. Hatcher's book~\cite[\S4.2,
Examples 4.53, 4.54 and 4.55]{MR1867354} for the reference. We treat $G(n,m)$ as
a~metric space with the following metric
\begin{defin}
  Let $U,V \in G(n,m)$. We define the metric
  \begin{displaymath}
    \dgras(U,V) = \| \proj_U - \proj_V \| = \sup_{w \in \Sphere} | \proj_U(w) - \proj_V(w) | \,.
  \end{displaymath}
\end{defin}
Note that this metric is different from the~geodesic distance
on~the~Grassmannian. However, the topology induced by~the~metric $\dgras$ agrees
with the standard quotient topology which is~the~same as~the~topology induced
by~the~geodesic distance.
\begin{rem}
  Let $I : \R^n \to \R^n$ denote the identity mapping. We will frequently use
  the following identity without reference
  \begin{displaymath}
    \dgras(U,V) 
    = \| \proj_U - \proj_V \| 
    = \| I - \pperp_U - (I - \pperp_V) \| 
    = \| \pperp_V - \pperp_U \| \,.
  \end{displaymath}
\end{rem}

\begin{defin}
  Let $V \in G(n,m)$ and let $(v_1,\ldots,v_m)$ be the basis of~$V$. Fix some
  radius $\rho > 0$ and a~small constant $\varepsilon \in (0,1)$ We say that
  $(v_1,\ldots,v_m)$ is a~\emph{$\red$-basis} if
  \begin{displaymath}
    \forall i,j \in \{ 1,\ldots, m \}
    \qquad
    (\delta_i^j - \varepsilon) \rho^2 
    \le |\langle v_i, v_j \rangle| \le
    (\delta_i^j + \varepsilon) \rho^2 \,.
  \end{displaymath}
  Here $\delta_i^j$ denotes the~Kronecker delta.
\end{defin}

\begin{prop}
  \label{prop:red-ang}
  Let $(v_1,\ldots,v_m)$ be a~$\red$-basis of~$V \in G(n,m)$ with constants
  $\rho = \rho_0 > 0$ and~$\varepsilon = \varepsilon_0 \in (0,1)$. Let
  $(u_1,\ldots,u_m)$ be some basis of~$U \in G(n,m)$, such that $|u_i - v_i| \le
  \vartheta \rho_0$ for some $\vartheta > 0$ and for each $i =
  1,\ldots,m$. There exist constants $C_{\red} = C_{\red}(m)$ and $\ere =
  \ere(m)$ such that whenever $\varepsilon_0 \le \ere$, then
  \begin{displaymath}
    \dgras(U,V) \le C_{\red} \vartheta \,.
  \end{displaymath}
\end{prop}

\begin{lem}
  \label{lem:GS}
  Let $(v_1,\ldots,v_m)$ be a~$\red$-basis of~$V \in G(n,m)$ with constants
  $\rho = \rho_0 = 1$ and~$\varepsilon = \varepsilon_0 \in (0,1)$. There exists
  an orthonormal basis $\hat{v}_1,\ldots,\hat{v}_m$ of $V$ and a constant
  $C_{gs} = C_{gs}(m)$ such that $|v_i - \hat{v}_i| \le C_{gs} \varepsilon_0$.
\end{lem}

\begin{proof}
  Set
  \begin{displaymath}
    \hat{v}_1 = \frac{v_1}{|v_1|} \,,\quad
    s_i = \sum_{j=1}^{i-1} \langle v_i, \hat{v}_j \rangle \hat{v}_j \,, \quad
    \tilde{v}_i = v_i - s_i
    \quad \text{and} \quad
    \hat{v}_i = \frac{\tilde{v}_i}{|\tilde{v}_i|} \,.
  \end{displaymath}
  We proceed by induction. For $i = 1$, we have $|\hat{v}_1 - v_1| = |1 - |v_1||
  \le \varepsilon_0$. Assume that for $i = 1, \ldots, i_0-1$ we have $|\hat{v}_i
  - v_i| = C \varepsilon_0$ for some constant $C = C(i)$. It follows that
  \begin{displaymath}
    |v_{i_0} - \tilde{v}_{i_0}| = |s_{i_0}|
    \le \sum_{j=1}^{i_0-1} |\langle v_{i_0}, v_j \rangle| + |\langle v_{i_0}, \hat{v}_j - v_j \rangle |
    \le \hat{C}(i_0) \varepsilon_0 
  \end{displaymath}
  \begin{displaymath}
    \text{and} \quad
    1 - (\hat{C}(i_0)+1) \varepsilon_0
    \le |v_{i_0}| - |s_{i_0}| 
    \le |\tilde{v}_{i_0}| 
    \le |v_{i_0}| + |s_{i_0}| 
    \le 1 + (\hat{C}(i_0)+1) \varepsilon_0 \,,
  \end{displaymath}
  \begin{displaymath}
    \text{hence}\quad
    |v_{i_0} - \hat{v}_{i_0}|
    \le |v_{i_0} - \tilde{v}_{i_0}| + |\tilde{v}_{i_0} - \hat{v}_{i_0}|
    \le (2\hat{C}(i_0) + 1) \varepsilon_0 \,. \qedhere
  \end{displaymath}
\end{proof}

\begin{lem}
  \label{lem:proj-ang}
  Let $(\hat{v}_1,\ldots,\hat{v}_m)$ be an~orthonormal basis of~$V \in G(n,m)$
  and let $U \in G(n,m)$ be such that $|\pperp_U(\hat{v}_i)| \le
  \vartheta$. There exists a~constant $C_{\pi} = C_{\pi}(m)$ such that
  $\dgras(U,V) \le C_{\pi} \vartheta$.
\end{lem}

\begin{proof}
  Without loss of generality, we can assume that $\vartheta < 1$. If $\vartheta
  \ge 1$ then we can set $C_{\pi} = 2$ and there is nothing to prove. Set $u_i =
  \proj_{U} \hat{v}_i$. Since $\langle \hat{v}_i, \hat{v}_j \rangle = 0$ for $i
  \ne j$, we have
  \begin{displaymath}
    \langle u_i, u_j \rangle
    = \langle \pperp_{U} \hat{v}_i, \pperp_{U} \hat{v}_j \rangle
    \quad \text{and} \quad
    \delta_i^j - \vartheta^2 \le |\langle u_i, u_j \rangle| \le \delta_i^j + \vartheta^2 \,,
  \end{displaymath}
  so $u_1,\ldots,u_m$ is a $\red$-basis with $\rho = 1$ and $\varepsilon =
  \vartheta^2$. From Lemma~\ref{lem:GS} there exists an~orthonormal basis
  $\hat{u}_1,\ldots,\hat{u}_m$ such that $|u_i - \hat{u}_i| \le C_{gs}
  \vartheta^2$. Hence $|\hat{v}_i - \hat{u}_i| \le C_{gs} \vartheta^2 +
  \vartheta \le (1+C_{gs})\vartheta$.

  We calculate
  \begin{align}
    \label{est:ortho-ortho}
    \dgras(U,V) &= \sup_{w \in \Sphere} |\pi_U(w) - \pi_V(w)|
    = \sup_{w \in \Sphere} \left| \sum_{i=1}^m \langle w, \hat{u}_i \rangle \hat{u}_i 
      - \langle w, \hat{v}_i \rangle \hat{v}_i \right| \\
    &\le \sup_{w \in \Sphere} \sum_{i=1}^m |\langle w, \hat{u}_i \rangle (\hat{u}_i - \hat{v}_i)|
      + |\langle w, (\hat{u}_i - \hat{v}_i) \rangle \hat{v}_i|
    \le 2 m (1+C_{gs}) \vartheta \,. \notag \qedhere
   \end{align}
\end{proof}

\begin{proof}[Proof of Proposition~\ref{prop:red-ang}]
  Dividing each $v_i$ by $\rho_0$, we get a $\red$-basis with $\rho = 1$. Hence
  we~can assume that $\rho_0 = 1$. Without loss of generality we may also assume
  that $\vartheta < 1$. Indeed, we always have the trivial estimate $\dgras(U,V)
  \le 2$, so if $\vartheta \ge 1$ we can set $C_{\red} = 2$.

  Let $\hat{v}_1,\ldots,\hat{v}_m$ be the~orthonormal basis given
  by~Lemma~\ref{lem:GS} applied to $v_1,\ldots,v_m$. Then
  \begin{displaymath}
    |\pperp_U \hat{v}_i|
    \le |\pperp_U(\hat{v}_i - v_i)| + |\pperp_U v_i|
    \le |\hat{v}_i - v_i| \dgras(U,V) + |v_i - u_i|
    \le C_{gs} \varepsilon_0 \dgras(U,V) + \vartheta
  \end{displaymath}
  for each $i = 1,\ldots,m$. We set $\ere = \ere(m) = \frac 12 (C_{\pi}
  C_{gs})^{-1}$ and we assume $\varepsilon_0 \le \ere$. Applying
  Lemma~\ref{lem:proj-ang} we obtain the estimate
  \begin{displaymath}
    \dgras(U,V) \le C_{\pi} C_{gs} \varepsilon_0 \dgras(U,V) + C_{\pi} \vartheta 
    \quad \iff \quad
    \dgras(U,V) \le \frac{C_{\pi}}{1 - C_{\pi} C_{gs} \varepsilon_0} \vartheta 
    \,. \qedhere
  \end{displaymath}
\end{proof}



\section{Geometric Morrey-Sobolev embedding}
\label{sec:first-reg}

In this section we prove Theorem~\ref{thm:regularity} which is a~geometric
counterpart of~the Morrey-Sobolev embedding $W^{2,p}(\R^k) \subseteq
C^{1,1-k/p}$ for $p > k$. We also give some examples of $m$-fine sets to which
Theorem~\ref{thm:regularity} applies.

\subsection{Proof of Theorem~\ref{thm:regularity}}
\label{sec:bdd-ene-flat}

\begin{prop}
  \label{prop:eta-d-balance}
  Let $l \in \{1,2,\ldots,m+2\}$ and $p > ml$. Assume $\Sigma \subseteq \R^n$
  satisfies \eqref{def:ahlfors} and also $\E_p^l(\Sigma) \le E < \infty$.
  Let~$T_0 = (x_0,\ldots,x_{m+1}) \in \Sigma^{m+2}$. If $\mathbf{T}_0 = \simp
  T_0$ is $(\eta,d)$-voluminous with $d \le R_{\ahl}$, then $\eta$ and~$d$ must
  satisfy
  \begin{equation}
    \label{est:eta-d}
    d \ge \left( \frac{C_{\eta d} A_{\ahl}^l}{E} \right)^{1/\lambda} \eta^{\kappa/\lambda}
    \quad \text{or equivalently} \quad
    \eta \le \left( \frac{E}{C_{\eta d} A_{\ahl}^l} \right)^{1/\kappa} d^{\lambda/\kappa} \,,
  \end{equation}
  where $C_{\eta d} = C_{\eta d}(m,l,p)$ is some constant, $\lambda = p - ml$
  and $\kappa = (p+ml)(m+1)$.
\end{prop}

\begin{proof}
  We shall estimate the $\E_p^l$-energy of~$\Sigma$. Recall that
  $\varsigma_{m+1} \le \frac 14$ was defined by~\eqref{def:varsigma}.
  \begin{multline}
    \label{eq:energy-low}
    \infty > E \ge \E_p^l(\Sigma) 
    = \int_{\Sigma^l} \sup_{y_l,\ldots,y_{m+1} \in \Sigma} \K^p(y_0,\ldots,y_{m+1})\ d\HM^{ml}_{(y_0,\ldots,y_{l-1})} \\
    \ge \int_{\Sigma \cap \Ball(x_0,\varsigma_{m+1} d)} \cdots \int_{\Sigma \cap \Ball(x_{l-1},\varsigma_{m+1} d)}
    \sup_{y_l,\ldots,y_{m+1} \in \Sigma} \K^p(\simp(y_0,\ldots,y_{m+1}))\ d\HM^{ml}_{(y_0,\ldots,y_{l-1})} \,.
  \end{multline}
  Proposition~\ref{prop:vol-close} combined with Remark~\ref{rem:reg-curv}
  lets~us estimate the integrand
  \begin{displaymath}
    \sup_{y_l,\ldots,y_{m+1} \in \Sigma} \K^p(\simp(y_0,\ldots,y_{m+1})) 
    \ge \left( \frac{3\eta^{m+1}}{4(m+1)! d} \right)^p \,.
  \end{displaymath}
  Since~$\Sigma$ satisfies~\eqref{def:ahlfors}, we get a~lower bound
  on~the~measure of~the sets over which we integrate
  \begin{displaymath}
    \HM^m(\Sigma \cap \Ball(x_i,\varsigma_{m+1} d)) \ge A_{\ahl} (\varsigma_{m+1} d)^m \,.
  \end{displaymath}
  Plugging the last two estimates into \eqref{eq:energy-low} and recalling
  \eqref{eq:sigma-asymp} we obtain
  \begin{displaymath}
    E \ge \left( A_{\ahl} (\varsigma_{m+1} d)^m \right)^l
    \left( \frac{3\eta^{m+1}}{4(m+1)! d} \right)^p 
    = C_{\eta d}(m,l,p) A_{\ahl}^l d^{ml - p} \eta^{(p+ml)(m+1)} \,. \qedhere
  \end{displaymath}
\end{proof}

Proposition~\ref{prop:eta-d-balance} is interesting in~itself. It says that
whenever the energy of~$\Sigma$ is finite, we cannot have very small and
voluminous simplices with vertices on $\Sigma$. It gives a~bound on the
''regularity'' (i.e. parameter $\eta$) of~any simplex in~terms of~its diameter
$d$ and we see that $\eta$ goes to $0$ when we decrease $d$. Now we are ready to
prove Proposition~\ref{prop:beta-est}.

\begin{proof}[Proof of Proposition~\ref{prop:beta-est}]
  Fix some point $x \in \Sigma$ and a~radius $r \in (0,R_{\ahl})$. Let
  $\mathbf{T} = \simp T = \simp(x_0, \ldots, x_{m+1})$ be an $(m+1)$-simplex
  such that $x_i \in \Sigma \cap \CBall(x,r)$ for $i = 0,1,\ldots,m+1$ and such
  that $\mathbf{T}$ has maximal $\HM^{m+1}$-measure among all simplices with
  vertices in~$\Sigma \cap \CBall(x,r)$.
  \begin{displaymath}
    \HM^{m+1}(\mathbf{T}) = \max\{ \HM^{m+1}(\simp(x_0',\ldots,x_{m+1}')) : x_i' \in \Sigma \cap \CBall(x,r) \} \,.
  \end{displaymath}
  The existence of~such simplex follows from the fact that the set $\Sigma \cap
  \CBall(x,r)$ is compact and from the fact that the function $T \mapsto
  \HM^{m+1}(\simp T)$ is continuous with respect to $x_0$, \ldots, $x_{m+1}$.

  Rearranging the vertices of~$\mathbf{T}$ we can assume that $\hmin(\mathbf{T}) =
  \height_{m+1}(\mathbf{T})$, so the largest $m$-face of~$\mathbf{T}$ is
  $\simp(x_0,\ldots,x_m)$. Let $H = \opspan\{ x_1-x_0, \ldots, x_m-x_0 \}$, so
  that $x_0 + H$ contains the largest $m$-face of~$\mathbf{T}$. Note that the
  distance of~any point $y \in \Sigma \cap \CBall(x,r)$ from the affine plane
  $x_0 + H$ has to be less then or equal to $\hmin(\mathbf{T}) =
  \dist(x_{m+1},x_0+H)$. If we could find a~point $y \in \Sigma \cap
  \CBall(x,r)$ with $\dist(y,x_0+H) > \hmin(\mathbf{T})$, than the simplex
  $\simp(x_0,\ldots,x_m,y)$ would have larger $\HM^{m+1}$-measure than $\mathbf{T}$
  but this is impossible due to the choice of~$\mathbf{T}$.

  Since $x \in \Sigma \cap \CBall(x,r)$, we know that
  $\dist(x,x_0+H) \le \hmin(\mathbf{T})$, so we obtain
  \begin{equation}
    \label{est:beta-hmin}
    \forall y \in \Sigma \cap \CBall(x,r) \quad
    \dist(y,x+H) \le 2 \hmin(\mathbf{T}) \,.
  \end{equation}
  Now we only need to estimate $\hmin(\mathbf{T}) = \height_{m+1}(\mathbf{T})$ from
  above. Of~course $\mathbf{T}$ is $(\hmin(\mathbf{T})/(2r),2r)$-voluminous, so
  applying Proposition~\ref{prop:eta-d-balance} we obtain
  \begin{equation}
    \label{est:hmin-r}
    \frac{\hmin(\mathbf{T})}{2r} \le
    \left( \frac{E}{C_{\eta d} A_{\ahl}^l} \right)^{1/\kappa} (2r)^{\lambda/\kappa} \,.
  \end{equation}
  Putting~\eqref{est:beta-hmin} and~\eqref{est:hmin-r} together we get
  \begin{displaymath}
    \nbeta(x,r) \le \frac{2 \hmin(\mathbf{T})}{r}
    \le 4 \left( \frac{E}{C_{\eta d} A_{\ahl}^l} \right)^{1/\kappa} (2r)^{\lambda/\kappa}
    = C(m,l,p) \left( \frac{E}{A_{\ahl}^l} \right)^{1/\kappa} r^{\lambda/\kappa} \,.
    \qedhere
  \end{displaymath}
\end{proof}

Having Proposition~\ref{prop:beta-est} at our disposal we can easily prove
Theorem~\ref{thm:regularity}.

\begin{proof}[Proof of Theorem~\ref{thm:regularity}]
  We know already that $\nbeta(x,r) \le C(m,l,p,A_{\ahl},E) r^{\lambda/\kappa}$
  for $r < R_{\ahl}$. We assumed \eqref{def:theta-beta}, so $\Sigma$ is
  Reifenberg flat with vanishing constant. We~finish the~proof by~applying
  Proposition~\ref{prop:dkt-reg}.
\end{proof}

\subsection{Examples of $m$-fine sets}
\label{sec:fine-sets}

Here we give a~few examples of~$m$-fine sets.

\begin{ex}
  Let $M$ be any $m$-dimensional, compact, closed manifold of~class $C^1$ and
  let $f : M \to \R^n$ be an immersion. Then the image $\Sigma = \opim(f)$ is
  an~$m$-fine set. At~each point $x \in M$, there is a~radius $R_x$ such that
  the neighborhood $U_x \subseteq f^{-1}(\Ball(f(x),R_x))$ of~$x$ in~$M$ is
  mapped to the set $V_x = f(U_x) \subseteq \Ball(f(x),R_x)$ and is a~graph
  of~some Lipschitz function $\Phi_x : Df(x)T_xM \to (Df(x)T_xM)^{\perp}$. If
  we~choose $R_x$ small then we can make the Lipschitz constant of~$\Phi_x$
  smaller than some $\varepsilon > 0$. Due to compactness of~$M$ and continuity
  of~$Df$ we can find a~global radius $R_{\Sigma} = \min\{ R_x : x \in M
  \}$. Then we can safely set $A_{\ahl} = \sqrt{1 - \varepsilon^2}$ and
  $M_{\Sigma} = 4$.
\end{ex}

\begin{ex}
  Let $\Sigma$ be the van Koch snowflake in~$\R^2$. Then $\Sigma$ is $1$-fine
  but it fails to be $1$-dimensional.
\end{ex}

\begin{figure}[!htb]
  \centering
  \includegraphics{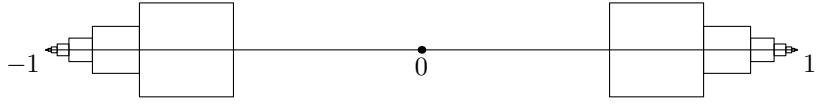}
  \caption{This set is $1$-fine despite the fact that it has boundary points.}
  \label{F:oscillation}
\end{figure}

\begin{ex}
  Let $m = 1$, $n = 2$ and
  \begin{displaymath}
    \Sigma = \bigcup_{k=1}^{\infty} (-Q_k) \cup \left\{
      (t,0) \in \R^2 : t \in [-1,1]
    \right\} \cup \bigcup_{k=1}^{\infty} Q_k \,,
  \end{displaymath}
  where
  \begin{displaymath}
    Q_0 = \partial \big([0,1] \times [0,1]\big)
    \quad \text{and} \quad
    Q_k = \Big( \sum_{j=1}^k 2^{-j}, -\tfrac 12 \Big) + 2^{-(k+1)} Q_0 \,.
  \end{displaymath}
  See Figure~\ref{F:oscillation} for a~graphical presentation. Condition
  \ref{def:theta-beta} holds at~the boundary points $(-1,0)$ and $(1,0)$
  of~$\Sigma$, because the $\beta$-numbers do not converge to zero with $r \to
  0$ at~these points. All the other points of~$\Sigma$ are internal points
  of~line segments or corner points of~squares, so at~these points condition
  \eqref{def:theta-beta} is also satisfied. Hence, $\Sigma$ is $1$-fine.

  This example shows that condition \eqref{def:theta-beta} does not exclude
  boundary points but at~any such boundary point we have to add some
  oscillation, to prevent $\beta$-numbers from getting too small. The same
  effect can be observed in~the following example
  \begin{displaymath}
    \Sigma = \partial \big([1,2] \times [-1,1]\big) \cup
    \overline{ \left\{ (x, x \sin( \tfrac 1x)) : x \in (0,1] \right\} } \,.
  \end{displaymath}
\end{ex}


\section{Uniform Ahlfors regularity - the proof of Theorem~\ref{thm:uahlreg}}
\label{sec:adm-sets}

Here we give the proof of~Theorem~\ref{thm:uahlreg}. First we introduce the
class of admissible sets, which is tailored for proving the existence of many
voluminous simplices (cf. Proposition~\ref{prop:big-proj-fat-simp}) with
vertices on~$\Sigma$. Proposition~\ref{prop:big-proj-fat-simp} is crucial in the
proof of~Theorem~\ref{thm:uahlreg}. In the end we also show how to make all the
emerging constants depend solely on $E$, $m$, $l$ and $p$.

\subsection{The class of admissible sets}

In~this section we introduce the definition of~the~class $\A(\delta,m)$
of~\emph{$(\delta,m)$-admissible} sets - here $\delta \in (0,1)$ is some
number. This definition is essentially the same
as~\cite[Definition~2.9]{1102.3642} but it is more convenient for~us to~impose
only local lower Ahlfors regularity~\eqref{def:ahlfors} instead of~Condition~H1
of~\cite[Definition~2.9]{1102.3642}.

\begin{defin}
  Let $I$ be a~countable set of~indices and assume there exist compact, closed,
  $m$-dimensional manifolds $M_i$ of~class $C^1$, a~set $Z$ with $\HM^m(Z) = 0$
  and continuous maps $f_i : M_i \to \R^n$ for $i \in I$, such that
  \begin{displaymath}
    \Sigma = \bigcup_{i \in I} f_i(M_i) \cup Z \,.
  \end{displaymath}
  Let $N$ be an~$(n-m)$-dimensional, compact, closed submanifold of $\R^n$.
  We~say that $\Sigma$ is \emph{linked with} $N$ and write $\oplk(\Sigma,N) =
  1$, if there exists an $i \in I$ such that the map
  \begin{displaymath}
    F : M_i \times N \to \Sphere^{n-1} \,, \quad
    F(w,z) = \frac{f_i(w) - z}{|f_i(w) - z|}
    \ \ \text{satisfies } \deg_2 F = 1 \,,
  \end{displaymath}
  where $\deg_2$ is the \emph{topological degree modulo~2}.
\end{defin}
For the definition of the~degree of a map we refer the reader
to~\cite[Chapter~5, \S~1]{MR0448362}.

\begin{defin}[cf.~\cite{1102.3642} Definition~2.9]
  \label{def:adm}
  Let $\delta \in (0,1)$ and let $I$ be a~countable set of~indices. Let $\Sigma$
  be a~compact subset of~$\R^n$ satisfying~\eqref{def:ahlfors}. We say that
  $\Sigma$ is $(\delta,m)$-\emph{admissible} and write $\Sigma \in \A(\delta,m)$
  if the following conditions are satisfied
  \begin{enumerate}[label=\texttt{A\arabic*}]
  \item \textbf{Mock tangent planes and flatness.}
    \label{adm:cone}
    There exists a~dense subset $\Sigma^* \subseteq \Sigma$ of full measure
    in~$\Sigma$ (i.e.~$\HM^m(\Sigma \setminus \Sigma^*) = 0$) such that for each
    $x \in \Sigma^*$ there is an $m$-plane $H = H_x \in G(n,m)$ and a~radius
    $r_0 = r_0(x) > 0$ such that
    \begin{displaymath}
      |\pperp_H(y-x)| < \delta |y-x| \quad \text{for each } y \in \Ball(x,r_0) \cap \Sigma \,.
    \end{displaymath}
  \item \textbf{Structure and linking.}
    \label{adm:str-link}
    There exist compact, closed, $m$-dimensional manifolds $M_i$ of~class $C^1$,
    a~set $Z$ with $\HM^m(Z) = 0$ and continuous maps $f_i : M_i \to \R^n$ for
    $i \in I$, such that
    \begin{displaymath}
      \Sigma = \bigcup_{i \in I} f_i(M_i) \cup Z 
    \end{displaymath}
    \begin{equation}
      \label{eq:adm:link}
      \text{and} \quad
      \forall x \in \Sigma^* \ \ 
      \oplk(\Sigma,\Slk_x) = 1 \ \ 
      \text{where}\ \  \Slk_x = \Sphere\left(x,\tfrac 12 r_0\right) \cap (x + H_x^{\perp}) \,.
    \end{equation}
  \end{enumerate}
\end{defin}

Condition \ref{adm:cone} ensures that at every point $x \in \Sigma^*$ one can
touch $\Sigma$ with an~apropriate cone. Condition \ref{adm:str-link} says that
at~each point of~$\Sigma$ there is a~sphere $\Slk_x$ which is linked with
$\Sigma$. This means intuitively, that we cannot move $\Slk_x$ far away from
$\Sigma$ without tearing one of~these sets. Example \ref{ex:flat} shows that
this condition is unavoidable for the theorems stated in~this paper to be true.

There are three especially useful properties of $\oplk$ that we want to use.

\begin{prop}[cf.~\cite{1102.3642}, Lemma 3.2]
  \label{prop:lk-htp-inv}
  Let $A \subseteq \R^n$ be a~$(\delta,m)$-admissible set and let $N$ be a
  compact, closed $(n-m-1)$-dimensional manifold of class $C^1$, and let $N_j =
  h_j(N)$ for $j=0,1$, where $h_j$ is a $C^1$ embedding of $N$ into $\R^n$ such
  that $N_j \cap \Sigma = \emptyset$. If there is a homotopy
  \begin{displaymath}
    G : N \times [0,1] \to \R^n \setminus \Sigma \,,
  \end{displaymath}
  such that $G(-,0) = h_0$ and $G(-,1) = h_1$, then 
  \begin{displaymath}
    \oplk(\Sigma,N_0) = \oplk(\Sigma,N_1) \,.
  \end{displaymath}
\end{prop}

\begin{prop}[cf.~\cite{1102.3642}, Lemma 3.4]
  \label{prop:lk-far}
  Let $\Sigma \subseteq \R^n$ be a~$(\delta,m)$-admissible set. Chose $y \in
  \R^n$ and $\varepsilon \in \R$ such that $0 < \varepsilon < r < 2 \varepsilon$
  and $\dist(y,\Sigma) \ge 3 \varepsilon$. Then
  \begin{displaymath}
    \oplk(\Sigma,\Sphere(y,r) \cap (y+V)) = 0
  \end{displaymath}
  for each $V \in G(n,n-m)$.
\end{prop}

\begin{prop}[cf.~\cite{1102.3642}, Lemma 3.5]
  \label{prop:intpoint}
  Let $\Sigma \subseteq \R^n$ be a~$(\delta,m)$-admissible set. Assume that for
  some $y \in \R^n$, $r > 0$ and $V \in G(n,n-m)$ we have
  \begin{displaymath}
    \oplk(\Sigma,\Sphere(y,r) \cap (y+V)) = 1 \,.
  \end{displaymath}
  Then the disk $\Ball(y,r) \cap (y+V)$ contains at least one point of $\Sigma$.
\end{prop}

\begin{ex}
  \label{ex:mfld}
  Let $\Sigma$ be any closed, compact, $m$-dimensional submanifold of~$\R^n$
  of~class $C^1$. Then $\Sigma \in \A(\delta,m)$ for any $\delta \in (0,1)$.

  It is easy to verify that $\Sigma \in \A(\delta,m)$. Take $M_1 = \Sigma$ and
  $f_1 = \opid_{M_1}$. The set $Z$ will be empty, so $\Sigma^* =
  \Sigma$. At~each point $x \in \Sigma$ we set $H_x$ to be the tangent space
  $T_x\Sigma$. Small spheres centered at~$x \in \Sigma$ and contained in~$x +
  H_x^{\perp}$ are linked with $\Sigma$; for the proof see
  e.g. \cite[pp. 194-195]{MR898008}.  Note that we do not assume orientability;
  that is why we used degree modulo $2$.
\end{ex}

\begin{ex}
  \label{ex:union-mfld}
  Let $\Sigma = \bigcup_{i=1}^N \Sigma_i$, where $\Sigma_i$ are closed, compact,
  $m$-dimensional submanifolds of~$\R^n$ of~class $C^1$. Moreover assume that
  these manifolds intersect only on sets of~zero $m$-dimensional Hausdorff
  measure, i.e.
  \begin{displaymath}
    \HM^m(\Sigma_i \cap \Sigma_j) = 0 \quad \text{for } i \ne j \,.
  \end{displaymath}
  Then $\Sigma \in \A(\delta,m)$ for any $\delta \in (0,1)$.
\end{ex}

\begin{rem}
  Any $C^1$-manifold is $(\delta,m)$-admissible (cf. Example~\ref{ex:mfld}) for
  any $\delta \in (0,1)$, hence any $m$-fine set with finite $\E_p^l$-energy for
  some $p > ml$ is also $(\delta,m)$-admissible.

  It turns out that any $(\delta,m)$-admissible set with finite $\E_p^l$-energy
  for some $p > ml$ is also $m$-fine. We will not use this fact in this article.
  The proof for the $\E_p^{m+2}$-energy can be found
  in~\cite[Theorem~2.13]{slawek-phd}.

  If we do not assume finiteness of~the $\E_p^l$-energy then these two classes
  of sets are different and none of them is contained in the other.
\end{rem}

\begin{ex}
  Let 
  \begin{displaymath}
    \Sigma = \big( [0,1] \times \{ 0 \} \big)
    \cup \big( \{ 1 \} \times [0,1] \big)
    \cup \big( \{ (x,x^2) : x \in [0,1] \} \big) \subseteq \R^2 \,.
  \end{displaymath}
  Then $\Sigma$ is $(\delta,1)$-admissible for any $\delta \in (0,1)$ but it is
  not $1$-fine. It does not satisfy \eqref{def:theta-beta} at the points $(0,0)$
  and $(1,1)$.
\end{ex}

Now we give some negative examples showing the role
of~condition~\ref{adm:str-link}.

\begin{ex}
  \label{ex:flat}
  Let $H \in G(n,m)$ and let $\Sigma = \proj_H(\Sphere) = \Ball \cap H$. Then
  $\Sigma$ satisfies \eqref{def:ahlfors} and~condition~\ref{adm:cone} but it
  does not satisfy \ref{adm:str-link}. Hence, it is not admissible. Although
  $\Sigma$ is a~compact, $m$-dimensional submanifold of~$\R^n$ of~class $C^1$,
  it is not closed.
\end{ex}

\begin{ex}
  \label{ex:decomp}
  Let $\Sigma = \Sphere \cap \R^{m+1}$. Of~course $\Sigma$ is admissible as it
  falls into the case presented in~Example~\ref{ex:mfld}. We want to emphasize
  that there are good and bad decompositions of~$\Sigma$ into the sum $\bigcup
  f_i(M_i)$ from condition \ref{adm:str-link}.

  The easiest one and the best one is to set $M_1 = \Sigma$ and $f_1 =
  \opid_{M_1}$. But there are other possibilities. Set $M_1 = \Sphere \cap
  \R^{m+1}$ and $M_2 = \Sphere \cap \R^{m+1}$ and set
  \begin{align*}
    f_1(x_1,\ldots,x_{m+1}) &= (x_1,\ldots,x_m,|x_{m+1}|) \,, \\
    f_2(x_1,\ldots,x_{m+1}) &= (x_1,\ldots,x_m,-|x_{m+1}|) \,,
  \end{align*}
  so that $f_1$ maps $M_1$ to the upper hemisphere and $f_2$ maps $M_2$ to the
  lower hemisphere. This decomposition is bad, because~\eqref{eq:adm:link} is
  not satisfied at~any point.
\end{ex}

\subsection{Homotopies inside cones}

In~this section we prove a few useful facts about cones. In~the proof
of~Proposition~\ref{prop:big-proj-fat-simp} we construct a~set $F$ by glueing
conical caps together. Then we need to know that we can deform one sphere lying
in~$F$ to some other sphere lying in~$F$ without leaving $F$. To be able to do
this easily we need Propositions~\ref{prop:sph-trans} and~\ref{prop:two-cones}.

\begin{defin}
  \label{def:grass-cone}
  Let $H \in G(n,m)$ be an $m$-dimensional subspace of~$\R^n$ and let $\delta
  \in (0,1)$ be some number. We define the set
  \begin{displaymath}
    \G(\delta,H) =
    \{
    V \in G(n,n-m) : \forall v \in V \,\, |\pperp_H(v)| \ge \delta |v|
    \} \,.
  \end{displaymath}
\end{defin}

In~other words $V \in \G(\delta,H)$ if and only if $V$ is contained in~the cone
$C(\delta,H)$. If $n = 3$ and $m = 1$ then $H$ is a~line in~$\R^3$ and the cone
$C(\delta,H)$ contains all the $2$-dimensional planes $V$ such that
$\sin(\opang(H,V)) \ge \delta$.

\begin{prop}[cf.~\cite{slawek-phd} Proposition~4.2]
  \label{prop:grass-path}
  For any two spaces $U$ and $V$ in~$\G(\delta,H)$ there exists a~continuous
  path $\gamma : [0,1] \to \G(\delta,H)$ such that $\gamma(0) = V$ and
  $\gamma(1) = U$.
\end{prop}

\begin{cor}[cf.~\cite{slawek-phd} Corollary~4.3]
  \label{cor:path-lift}
  The path $\gamma$ from Proposition~\ref{prop:grass-path} lifts to a~continuous
  path $\tilde{\gamma} : [0,1] \to O(n)$ in~the orthogonal group.
\end{cor}

The proofs can be found in~\cite[Section~4.1.1]{slawek-phd}

\begin{cor}
  \label{cor:sph-in-cone}
  Let $H$ and $\delta$ be as in~Proposition~\ref{prop:grass-path}. Let $S_1$ and
  $S_2$ be two round spheres centered at~the origin, contained in~the conical
  cap $\Cone(\delta,H,\rho_1,\rho_2)$ and of~the same dimension
  $(n-m-1)$. Moreover assume that $0 \le \rho_1 < \rho_2 $.  There exists an
  isotopy
  \begin{displaymath}
    F : S_1 \times [0,1] \to \Cone(\delta,H,\rho_1,\rho_2) \,,
  \end{displaymath}
  \begin{displaymath}
    \text{such that} \quad
    F(-,0) = \opid_{S_1}
    \qquad \text{and} \qquad
    F(S_1 \times \{1\}) = S_2 \,. 
  \end{displaymath}
\end{cor}

\begin{proof}
  Let $r_1$ and $r_2$ be the radii of~$S_1$ and $S_2$ respectively. We have
  $\rho_1 < r_1,r_2 < \rho_2$. Let $V_1,V_2 \in G(n,n-m)$ be the two subspaces
  of~$\R^n$ such that $S_1 \subseteq V_1$ and $S_2 \subseteq V_2$. In~other
  words $S_1 = \Sphere_{r_1} \cap V_1$ and $S_2 = \Sphere_{r_2} \cap
  V_2$. Because $S_1$ and $S_2$ are subsets of~$\Cone(\delta,H)$, we know that
  $V_1$ and $V_2$ are elements of~$\G(\delta,H)$. From
  Proposition~\ref{prop:grass-path} we get a~continuous path $\gamma$ joining
  $V_1$ with $V_2$. By Corollary~\ref{cor:path-lift}, this path lifts to a~path
  $\tilde{\gamma}$ in~the orthogonal group $O(n)$. For $z \in S_1$ and $t \in
  [0,1]$ we set
  \begin{align*}
    F(z,t) &= \tilde{\gamma}(t) \tilde{\gamma}(0)^{-1} z \,.
  \end{align*}
  This gives a~continuous deformation of~$S_1 = \Sphere_{r_1} \cap V_1$ into
  $\Sphere_{r_1} \cap V_2$. Now, we only need to adjust the radius but this can
  be easily done inside $V_2 \cap \Shell(\rho_1,\rho_2)$ so the corollary
  is~proven.
\end{proof}

\begin{prop}
  \label{prop:sph-trans}
  Let $H \in G(n,m)$. Let $S$ be a~sphere perpendicular to $H$, meaning that $S
  = \Sphere(x,r) \cap (x + H^{\perp})$ for some $x \in H$ and $r > 0$. Assume
  that $S$ is contained in~the conical cap $\Cone(\delta,H,\rho_1,\rho_2)$,
  where $\rho_2 > 0$. Fix some $\rho \in (\rho_1,\rho_2)$. There exists an
  isotopy
  \begin{displaymath}
    F : S \times [0,1] \to \Cone(\delta,H,\rho_1,\rho_2) \,,
  \end{displaymath}
  \begin{displaymath}
    \text{such that} \quad
    F(\cdot,0) = \opid_{S} 
    \qquad \text{and} \qquad
    F(S \times \{1\}) = \Sphere_{\rho} \cap H^{\perp} \,.
  \end{displaymath}
\end{prop}

\begin{figure}[!htb]
  \centering
  \includegraphics[scale=0.7]{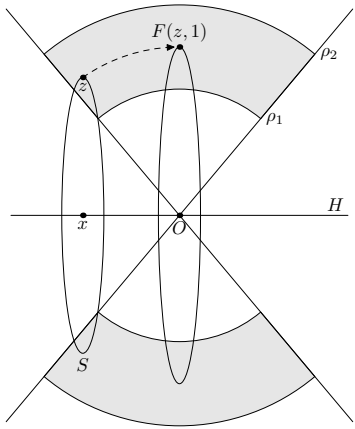}
  \caption{When we move the center of~a~sphere to the origin, we need to control
    the radius so that the deformation is performed inside the conical cap.}
  \label{F:sphere-move}
\end{figure}

\begin{proof}
  Any point $z \in S$ can be uniquely decomposed into a~sum $z = x + r y$, where
  $y \in \Sphere \cap H^{\perp}$ is a~point in~the unit sphere in~$H^{\perp}$. We define
  \begin{displaymath}
    F(x + r y, t) = (1-t)x + y \sqrt{r^2 + |x|^2 - |(1-t)x|^2} \,.
  \end{displaymath}
  This gives an isotopy which deforms $S$ to a~sphere perpendicular to $H$ and
  centered at~the origin (see Figure~\ref{F:sphere-move}). Fix some $z = x + ry
  \in S$. The sphere $S$ is contained in~$\Cone(\delta,H)$, so it follows that
  \begin{displaymath}
    \frac{|\pperp_H(F(z,t))|}{|F(z,t)|} 
    = \frac{\sqrt{r^2 + |x|^2 - |(1-t)x|^2}}{\sqrt{r^2 + |x|^2}}
    \ge \frac{r}{\sqrt{r^2 + |x|^2}}
    = \frac{|\pperp_H(z)|}{|z|} \ge \delta \,.
  \end{displaymath}
  This shows that the whole deformation is performed inside
  $\Cone(\delta,H)$. Next, we need to continuously change the radius to the
  value $\rho$ but this can be easily done inside $H^{\perp} \cap
  (\Ball_{\rho_2} \setminus \Ball_{\rho_1})$.
\end{proof}

Next, we give a sufficient condition on $\alpha$ and $\beta$ assuring that
$\Cone(\alpha,P) \cap \Cone(\beta,H)$ contains another cone $\Cone(\gamma,H)$
for some $\gamma \in (0,1)$. This allows to construct homotopies of spheres
inside $\Cone(\alpha,P) \cup \Cone(\beta,H)$

\begin{prop}
  \label{prop:two-cones}
  Let $\alpha > 0$ and $\beta > 0$ be two real numbers satisfying $\alpha+\beta <
  \sqrt{1-\beta^2}$ and let $H_0,H_1 \in G(n,m)$ be two $m$-planes in~$\R^n$.
  Assume that
  \begin{displaymath} 
    \Cone(\sqrt{1-\alpha^2},H_0^{\perp}) \cap \Cone(\sqrt{1-\beta^2},H_1^{\perp}) \ne \emptyset \,.
  \end{displaymath}
  Then for any $\epsilon > 0$ we have the inclusion
  \begin{equation}
    \label{eq:gen-cone-int}
    \Cone((\alpha+\beta)/\sqrt{1-\beta^2} + \epsilon,H_0)
    \subseteq
    \Cone(\epsilon,H_1)\,.
  \end{equation}
  In~particular, if $\alpha+\beta \le (1-\beta)\sqrt{1-\beta^2}$, then
  \begin{displaymath}
    H_0^{\perp} \subseteq \Cone(\alpha,H_0) \cap \Cone(\beta,H_1) \,.
  \end{displaymath}
\end{prop}

\begin{proof}
  First we estimate the ``angle'' between $H_0$ and $H_1$. Since the cones
  $\Cone(\sqrt{1-\alpha^2},H_0^{\perp})$ and
  $\Cone(\sqrt{1-\beta^2},H_1^{\perp})$ have nonempty intersection they both
  must contain a~common line $L \in G(n,1)$.
  \begin{displaymath}
    L \subseteq \Cone(\sqrt{1-\alpha^2},H_0^{\perp}) \cap \Cone(\sqrt{1-\beta^2},H_1^{\perp}) \,.
  \end{displaymath}
  Choose some point $z \in H_1$ and find a~point $y \in L$ such that $z =
  \proj_{H_1}(y)$. Since $y \in \Cone(\sqrt{1-\beta^2},H_1^{\perp})$ it follows
  that $|\pperp_{H_1}(y)| < \beta |y|$.  Furthermore, by the Pythagorean theorem
  \begin{displaymath}
    |y|^2 = |\proj_{H_1}(y)|^2 + |\pperp_{H_1}(y)|^2 
    \le |z|^2 + \beta^2|y|^2  \,,
    \quad \text{hence} \quad
    |y| \le \frac{|z|}{\sqrt{1 - \beta^2}} \,.
  \end{displaymath}
  Because $y$ also belongs to the cone $\Cone(\sqrt{1-\alpha^2},H_0^{\perp})$ we
  have $|\pperp_{H_0}(y)| < \delta |y|$, so we obtain
  \begin{align}
    |\pperp_{H_0}(z)| &\le |\pperp_{H_0}(y)| + |\pperp_{H_0}(z-y)| 
    \le |\pperp_{H_0}(y)| + |z-y|  \notag \\
    &= |\pperp_{H_0}(y)| + |\pperp_{H_1}(y)|
    \le \alpha |y| + \beta |y|  
    \le \frac{\alpha + \beta}{\sqrt{1-\beta^2}} |z|
    \qquad \text{for all } z \in H_1
    \label{est:QH0z} \,.
  \end{align}

  Choose some $\epsilon > 0$ and let
  \begin{displaymath}
    x \in C\left(
      \frac{\alpha+\beta}{\sqrt{1-\beta^2}} + \epsilon,H_0
    \right) \,,
    \quad \text{so} \quad
    |\pperp_{H_0}(x)| \ge \left(
      \frac{\alpha+\beta}{\sqrt{1-\beta^2}} + \epsilon
    \right)
    |x| \,.
  \end{displaymath}
  If $\epsilon$ is small enough, then such $x$ exists by the assumption that
  $\alpha+\beta < \sqrt{1-\beta^2}$. For bigger~$\epsilon$ the inclusion
  $\Cone((\alpha+\beta)/\sqrt{1-\beta^2} + \epsilon,H_0) \subseteq
  \Cone(\epsilon,H_1)$ is trivially true. From the triangle inequality
  \begin{align*}
    \frac{\alpha+\beta}{\sqrt{1-\beta^2}} |x|
    &\le |\pperp_{H_0}(x)| 
    \le |\pperp_{H_0}(\pperp_{H_1}(x))| + |\pperp_{H_0}(\proj_{H_1}(x))| \\
    &\le |\pperp_{H_1}(x)| + |\pperp_{H_0}(\proj_{H_1}(x))| \,,
  \end{align*}
  \begin{displaymath}
    \text{hence} \quad
    |\pperp_{H_1}(x)| \ge \frac{\alpha+\beta}{\sqrt{1-\beta^2}} |x| + \epsilon |x| - |\pperp_{H_0}(\proj_{H_1}(x))| \,.
  \end{displaymath}
  Because $\proj_{H_1}(x) \in H_1$ and because of~estimate \eqref{est:QH0z} we have
  \begin{displaymath}
    |\pperp_{H_1}(x)| 
    \ge \frac{\alpha+\beta}{\sqrt{1-\beta^2}} |x| + \epsilon |x| - \frac{\alpha+\beta}{\sqrt{1-\beta^2}} |\proj_{H_1}(x)|
    \ge \epsilon |x| \,. \qedhere
  \end{displaymath}
\end{proof}

\subsection{The construction of voluminous simplices}

For any $x_0 \in \Sigma^*$ Proposition~\ref{prop:big-proj-fat-simp} stated
below, ensures the existence of~$d = d(x_0) > 0$ and an~$(\eta,d)$-voluminous
simplex with vertices on $\Sigma \cap \CBall(x_0,d)$ and also that at~any scale
below $d$ our set $\Sigma$ has big projection onto some affine $m$-plane.
The~reasoning used here mimics \cite[Proposition~3.5]{0911.2095}. Note that,
finiteness of the~$\E_p^l$-energy is not used in the proof.

\begin{prop}
  \label{prop:big-proj-fat-simp}
  Let $\delta \in (0,1)$ and $\Sigma \in \A(\delta,m)$ be an admissible set.
  There exists an~$\eta_0 = \eta_0(\delta,m) > 0$ such that for every point $x_0
  \in \Sigma^*$ there is a~stopping distance $d = d(x_0) > 0$ and
  a~$(m+1)$-tuple of~points $(x_1, x_2, \ldots, x_{m+1}) \in \Sigma^{m+1}$ such
  that $\mathbf{T} = \simp(x_0, \ldots, x_{m+1})$ is
  $(\eta_0,d)$-voluminous. Moreover, for all $\rho \in (0,\frac 12 d)$ there
  exists an $m$-dimensional subspace $H = H(x_0,\rho) \in G(n,m)$ with the
  property
  \begin{equation}
    \label{eq:big-proj}
    (x_0 + H) \cap \Ball(x_0, \sqrt{1 - \delta^2} \rho)
    \subseteq \proj_{x_0 + H}(\Sigma \cap \Ball(x_0, \rho)) \,.
  \end{equation}
\end{prop}

\begin{cor}
  \label{cor:uahlreg}
  For any $x_0 \in \Sigma^*$ and any $\rho \le \frac 12 d(x_0)$ we have
  \begin{displaymath}
    \HM^m(\Sigma \cap \Ball(x_0,\rho)) \ge (1 - \delta^2)^{\frac m2} \omega_m \rho^m \,.
  \end{displaymath}
\end{cor}

\begin{proof}
  The orthogonal projection $\proj_{x_0 + H}$ is Lipschitz with constant $1$ so
  it cannot increase the $\HM^m$-measure. From~\eqref{eq:big-proj} we know that
  the image of~$\Sigma \cap \Ball(x_0, \rho)$ under $\proj_{x_0 + H}$ contains
  the ball $(x_0 + H) \cap \Ball(x_0, \sqrt{1 - \delta^2} \rho)$. The measure
  of~that ball equals $(1 - \delta^2)^{\frac m2} \omega_m \rho^m$.
\end{proof}

\begin{proof}[Proof of~Proposition~\ref{prop:big-proj-fat-simp}]
  Without loss of~generality we can assume that $x_0 = 0$ is the origin.
  To~prove the proposition we will construct finite sequences of
  \begin{itemize}
  \item compact, connected, centrally symmetric sets $F_0 \subseteq F_1 \subseteq \ldots \subseteq F_N$,
  \item $m$-dimensional subspaces $H_i \subseteq \R^n$ for $i = 0, 1, \ldots, N$,
  \item and of~radii $\rho_0 < \rho_1 < \cdots < \rho_N$.
  \end{itemize}
  For brevity, we define 
  \begin{displaymath}
    r_i = \sqrt{1 - \delta^2} \rho_i \,.
  \end{displaymath}
  The above sequences will satisfy the following conditions
  \begin{itemize}
  \item the interior of~$F_i$ is disjoint with $\Sigma$, i.e.
    \begin{equation}
      \label{cond:disj}
      \Sigma \cap \opint F_i = \emptyset \,,
    \end{equation}
  \item the radii grow geometrically, i.e.
    \begin{equation}
      \rho_{i+1} \ge 2 \rho_i \,, \label{cond:growth}
    \end{equation}
  \item for each $i \ge 0$ the set $F_{i+1}$ contains a~large conical cap, i.e.
    \begin{equation}
      \label{cond:cone}
      \Cone(\delta,H_{i+1},\tfrac 12 \rho_i,\rho_{i+1}) \subseteq F_{i+1} \,,
    \end{equation}
  \item all spheres $S$ centered at~$H_i \cap \Ball_{r_i}$, perpendicular to
    $H_i$ (i.e.~ $S \subseteq H_i^{\perp} + p$ for some $p \in \R^n$) and
    contained in~$F_i$ are linked with $\Sigma$, i.e.
    \begin{equation}
      \label{cond:link}
      \forall\, x \in H_i \cap \Ball_{r_i} \;
      \forall\, s > 0 \;
      \left(
        S = \Sphere(x,s) \cap (x + H_i^{\perp}) \subseteq F_i
        \quad \Rightarrow \quad
        \oplk(\Sigma, S) = 1
      \right) .
    \end{equation}
  \end{itemize}

  Let us define the first elements of~these sequences. We set $H_0 = H_{x_0}$,
  $\rho_0 = 0$ and $F_0 = \emptyset$. Next, we set
  \begin{displaymath}
    H_1 = H_0 \,, \quad
    \rho_1 = \inf \{ s > 0 : \Cone(\delta,H_0,0,s) \cap \Sigma \ne \emptyset \} 
    \quad \text{and} \quad
    F_1 = \Cone(\delta,H_1,0,\rho_1) \,.
  \end{displaymath}
  Directly from the definition of~an admissible set, we know that $\rho_1 > 0$,
  so the condition \eqref{cond:growth} is satisfied for $i = 0$. Conditions
  \eqref{cond:disj} and \eqref{cond:cone} are immediate for $i = 0$. Using
  Proposition~\ref{prop:sph-trans} one can deform any sphere $S$ from condition
  \eqref{cond:link} to the sphere $\Slk_x$ defined in~\ref{adm:str-link} of~the
  definition of~$\A(\delta,m)$. This shows that \eqref{cond:link} is satisfied
  for $i = 0$.

  We proceed by induction. Assume we have already defined the sets $F_i$,
  subspaces $H_i$ and radii $\rho_i$ for $i = 0,1, \ldots, I$. Now, we will show
  how to continue the construction.

  Let $(e_1,e_2, \ldots, e_m)$ be an orthonormal basis of~$H_I$. We choose $m$
  points lying on $\Sigma$ such that
  \begin{displaymath}
    x_i \in \Sigma \cap \Ball(r_I e_i, \delta \rho_I) \cap (H_I^{\perp} + r_I e_i) 
  \end{displaymath}
  \begin{equation}
    \text{and in~particular} \quad
    \label{eq:good-ball}
    x_i \in \Ball(x_0, 2\rho_I)
    \quad \text{for} \quad
    i \in \{0,1,\ldots,m\}\,.
  \end{equation}
  Condition \eqref{cond:link} together with Proposition~\ref{prop:intpoint}
  ensure that such points exist. The $m$-simplex $\mathbf{R} = \simp(x_0, x_1,
  \ldots, x_m)$ will be the base of~our $(m+1)$-simplex $\mathbf{T}$. Note that
  \begin{displaymath}
    \diam(\mathbf{R}) \le 4 \rho_I
    \quad \text{and} \quad
    \proj_{H_I}(\mathbf{R}) = \simp(0, r_I e_1, r_I e_2, \ldots, r_I e_m) \,,
    \quad \text{hence} \quad
    \HM^m(\mathbf{R}) \ge \frac{r_I^m}{m!} \,.
  \end{displaymath}

  Recall that $x_0 = 0$ and set $P = \opspan\{ x_1, x_2, \ldots, x_m \}$. It
  suffices to find one more point $x_{m+1} \in \Sigma$ such that the distance
  $\dist(x_{m+1}, P) \ge \tilde{\eta} \rho_I$ for some positive $\tilde{\eta}$.
  Indeed, if we set $\mathbf{T} = \simp(x_0,\ldots,x_{m+1})$, we have
  \begin{equation}
    \label{eq:good-eta}
    \hmin(\mathbf{T}) 
    = \frac{(m+1) \HM^{m+1}(\mathbf{T})}{\max \{ \HM^m(\face_i \mathbf{T}) \}_{i=0}^{m+1}}
    \ge \frac{\tilde{\eta} \rho_I (m+1) \HM^m(\mathbf{R})}{(4\rho_I)^m \omega_m}
    \ge (4\rho_I) \frac{\tilde{\eta}(1-\delta^2)^{\frac m2}}{\omega_m 4^{m+1} m!} \,.
  \end{equation}

  Choose a~small positive number $h_0 = h_0(\delta) \le \frac 12$ such that
  \begin{equation}
    \label{def:h0choice}
    \delta + 2h_0\delta \le (1 - 2h_0\delta) \sqrt{1 - (2h_0\delta)^2} \,.
  \end{equation}
  This is always possible because when we decrease $h_0$ to $0$ the left-hand
  side of~\eqref{def:h0choice} converges to $\delta < 1$ and the right-hand side
  converges to $1$. We need this condition to be able to apply
  Proposition~\ref{prop:two-cones} later on.
  \begin{rem}
    \label{rem:delta14}
    Note that if $\delta \le \frac 14$, we can set $h_0 = \frac 12$ because
    then
    \begin{displaymath}
      \delta + 2 h_0 \delta \le \tfrac 12
      \quad \text{and} \quad 
      (1 - 2 h_0 \delta) \sqrt{1 - (2 h_0 \delta)^2}
      \ge \tfrac 34 \tfrac{\sqrt{15}}{16} \ge \tfrac{9}{16} \,.
    \end{displaymath}
  \end{rem}

  There are two possibilities (see Figure~\ref{F:casesAB})
  \begin{enumerate}[label=(\Alph*)]
  \item \label{pos:good} there exists a~point $x_{m+1} \in \Sigma \cap
    \Shell(\tfrac 12 \rho_I,2 \rho_I)$ such that
    \begin{displaymath}
      \dist(x_{m+1},P) \ge h_0 \delta \rho_I \,,
    \end{displaymath}
  \item \label{pos:flat} $\Sigma$ is contained in~a~small neighborhood of~$P$, i.e.
    \begin{displaymath} 
      \Sigma \cap \Shell(\tfrac 12 \rho_I,2 \rho_I) \subseteq  P + \Ball_{h_0 \delta \rho_I}\,.
    \end{displaymath}
  \end{enumerate}

  \begin{figure}[!htb]
    \centering
    \includegraphics{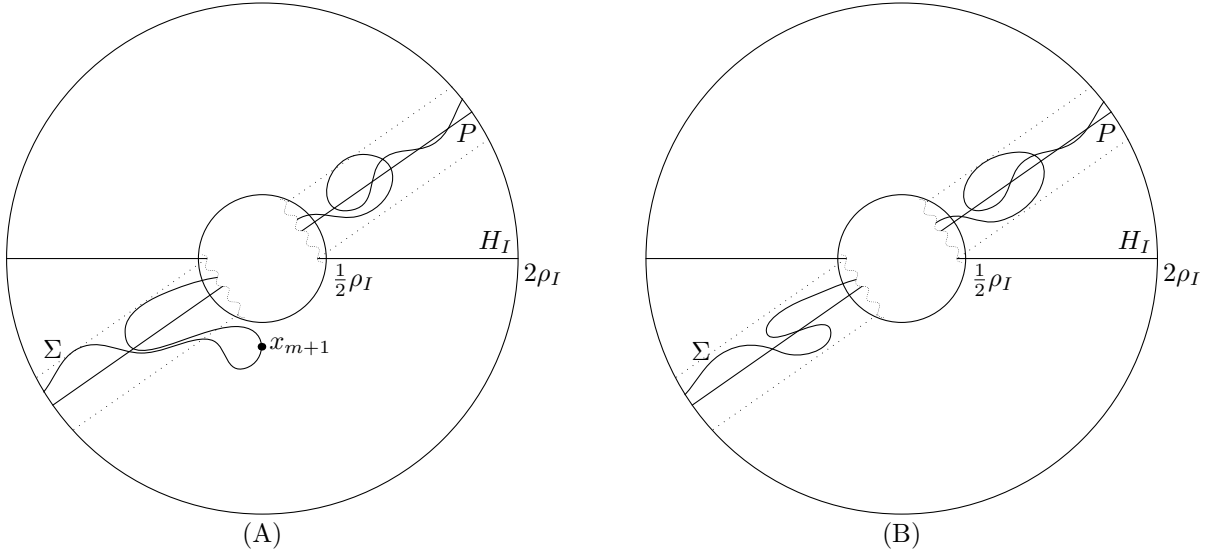}
    \caption{The two possible configurations.}
    \label{F:casesAB}
  \end{figure}

  If case \ref{pos:good} occurs, then we can end our construction immediately.
  The point $x_{m+1}$ satisfies
  \begin{displaymath}
    x_{m+1} \in \Ball(x_0, 2 \rho_I)
    \quad \text{and} \quad
    \dist(x_{m+1}, P) \ge h_0 \delta \rho_I \,.
  \end{displaymath}
  Hence, recalling~\eqref{eq:good-eta}, we may set
  \begin{equation}
    \label{def:eta} 
    \mathbf{T} = \simp(x_0, \ldots, x_{m+1}) \,, \quad
    N = I \,, \quad
    \eta_0 = \frac{h_0 \delta (1-\delta^2)^{\frac m2}}{\omega_m 4^{m+1} m!}
    \quad \text{and} \quad
    d = d(x_0) = 4 \rho_I \,. 
  \end{equation}

  If case \ref{pos:flat} occurs, then our set $\Sigma$ is almost flat
  in~$\Shell(\tfrac 12 \rho_I,2 \rho_I)$ so there is no chance of~finding
  a~voluminous simplex in~this scale and we have to continue our
  construction. Let
  \begin{itemize}
  \item $H_{I+1} = P$,
  \item $\rho_{I+1} = \inf \{ s > \rho_I : \Cone(\delta,P,\rho_I,s) \cap \Sigma \ne \emptyset \}$ and
  \item $F_{I+1} = F_I \cup \Cone(\delta,P,\tfrac 12 \rho_I,\rho_{I+1})$.
  \end{itemize}
  We assumed \ref{pos:flat}, so it follows that
  \begin{equation}
    \label{eq:S-in-cone}
    \forall x \in \Sigma \cap \Shell(\tfrac 12 \rho_I,2 \rho_I)
    \quad
    |\pperp_P(x)| \le h_0 \delta \rho_I \le 2 h_0 \delta |x| < \delta |x| \,.
  \end{equation}
  This means that $\Cone(\delta,P,\tfrac 12 \rho_I,2\rho_I)$ does not intersect
  $\Sigma$ and we can safely set $H_{I+1} = P$. It is immediate that $\rho_{I+1}
  \ge 2 \rho_I$ so conditions \eqref{cond:disj}, \eqref{cond:growth} and
  \eqref{cond:cone} are satisfied. Now, the only thing left is to verify condition
  \eqref{cond:link}.

  We are going to show that all spheres $S$ contained in~$F_{I+1}$ of~the form 
  \begin{displaymath}
    S = \Sphere(x,r) \cap (x + P^{\perp}) \,,
    \quad \text{for some } x \in P \cap \Ball_{r_{I+1}}
  \end{displaymath}
  are linked with $\Sigma$. By the inductive assumption, we already know that
  spheres centered at~$H_I \cap \Ball_{r_I}$, perpendicular to $H_I$ and
  contained in~$F_I$ are linked with $\Sigma$. Therefore, all we need to do is
  to continuously deform $S$ to an appropriate sphere centered at~$H_I$ and
  contained in~$F_I$ in~such a~way that we never leave the set $F_{I+1}$ (see
  Figure~\ref{F:htp-in-F}).

  \begin{figure}[!htb]
    \centering
    \includegraphics{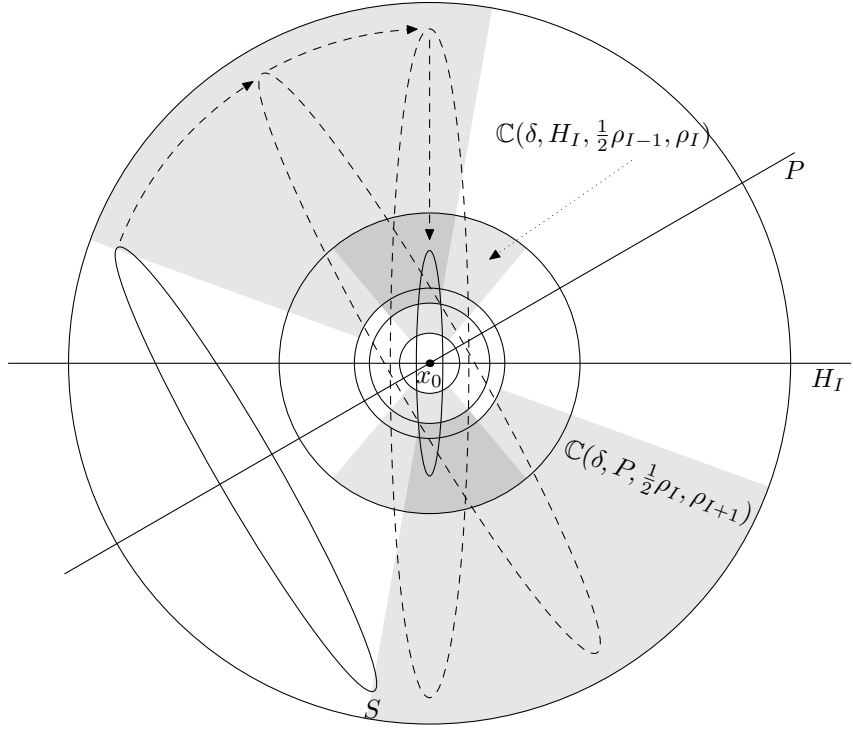}
    \caption{First we move the center of~$S$ to $x_0$. Then we rotate $S$ so that
      it is perpendicular to $H_I$. Finally we change the radius so that it is
      between $\frac 12 \rho_{I-1}$ and $\rho_I$.}
    \label{F:htp-in-F}
  \end{figure}

  We know that $F_{I+1}$ contains the conical cap $\mathbf{C} =
  \Cone(\delta,P,\tfrac 12 \rho_I,\rho_{I+1})$, so we can use
  Proposition~\ref{prop:sph-trans} to move $S$ inside $\mathbf{C}$, so that it
  is centered at~the origin.

  From \eqref{eq:S-in-cone} we get
  \begin{displaymath}
    \Sigma \cap \Shell(\tfrac 12 \rho_I,2 \rho_I)
    \subseteq
    \R^n \setminus \Cone(2 h_0 \delta, P)
    \subseteq
    \Cone(\sqrt{1 - (2 h_0 \delta)^2}, P^{\perp}) \,.
  \end{displaymath}
  Using this and our inductive assumption we obtain
  \begin{displaymath}
    \Sigma \cap \Shell(\tfrac 12 \rho_I,\rho_I)
    \subseteq
    \Cone(\sqrt{1 - \delta^2},H_I^{\perp}) \cap \Cone(\sqrt{1 - (2 h_0 \delta)^2},P^{\perp}) \,.
  \end{displaymath}

  We have two cones that have nonempty intersection and we chose $h_0$ such
  that~\eqref{def:h0choice} holds, so we can apply
  Proposition~\ref{prop:two-cones} with $\alpha = \delta$ and $\beta = 2 h_0
  \delta$. Hence the intersection $\Cone(\delta,H_I) \cap \Cone(\delta,P)$
  contains the space $H_I^{\perp}$. Therefore
  \begin{displaymath}
    H_I^{\perp} \cap \Shell(\tfrac 12 \rho_I,\rho_{I+1})
    \subseteq
    \Cone(\delta,P,\tfrac 12 \rho_I,\rho_{I+1}) \cap F_I \,.
  \end{displaymath}
  Using Corollary~\ref{cor:sph-in-cone} we can rotate $S$ inside $\mathbf{C}$,
  so that it lies in~$H^{\perp}$. Then we decrease the radius of~$S$ to the
  value e.g. $\tfrac 34 \rho_I \in (\frac 12 \rho_{I-1},\rho_I)$. Applying the
  inductive assumption we obtain condition \eqref{cond:link} for $i = I+1$.

  The set $\Sigma$ is compact and $\rho_i$ grows geometrically, so our
  construction has to end eventually. Otherwise we would find arbitrary large
  spheres, which are linked with $\Sigma$ but this contradicts compactness.
\end{proof}



\subsection{The proof of Theorem~\ref{thm:uahlreg}}

\begin{proof}[Proof of~Theorem~\ref{thm:uahlreg}]
  From Theorem~\ref{thm:regularity} we already know that $\Sigma$ is
  an~embedded, $C^{1,\lambda/\kappa}$-smooth manifold without boundary. Hence,
  it is also $(\delta,m)$-admissible for any $\delta \in (0,1)$
  (cf.~Example~\ref{ex:mfld}) and~$\Sigma^* = \Sigma$. Set $\delta = \frac 14$,
  then Corollary~\ref{cor:uahlreg} gives us Theorem~\ref{thm:uahlreg} where
  $R_0$ can be any number less than $d(\Sigma) = \inf_{x_0 \in \Sigma} d(x_0)$.
  Hence, it suffices to show that $d(\Sigma)$ can be bounded below independently
  of~$\Sigma$.

  From Proposition~\ref{prop:eta-d-balance} we know that $d(\Sigma)$ must
  satisfy \eqref{est:eta-d} with $\eta = \eta_0$ defined
  by~\eqref{def:eta}. Hence, we~already have a~positive lower bound on
  $d(\Sigma)$. We only need to show that it does not depend on $A_{\ahl}$.

  Fix a~point $x_0 \in \Sigma$ such that $d(x_0) < (1 + \varepsilon)d(\Sigma)$
  for some small $\varepsilon \in (0,1)$.
  Proposition~\ref{prop:big-proj-fat-simp} gives us an
  $(\eta_0,d(x_0))$-voluminous simplex $\simp(x_0, \ldots, x_{m+1})$. Recall
  that $\varsigma_{m+1} < \tfrac 14$ was defined by~\eqref{def:varsigma}. For
  each $i = 1, 2, \ldots, m+1$ we have
  \begin{displaymath}
    \varsigma_{m+1} d(x_0) \le \varsigma_{m+1} (1+\varepsilon) d(\Sigma) \le \frac 12 d(\Sigma) \le \frac 12 d(x_i) \,.
  \end{displaymath}
  Hence, applying Corollary~\ref{cor:uahlreg} we get
  \begin{displaymath}
    \HM^m(\Sigma \cap \Ball(x_i, \varsigma_{m+1} d(x_0)))
    \ge \frac{\sqrt{15^m}}{4^m} \omega_m (\varsigma_{m+1} d(x_0))^m \,.
  \end{displaymath}
  Now we can repeat the calculation from the proof
  of~Proposition~\ref{prop:eta-d-balance}, replacing $A_{\ahl}$ with
  $\frac{\sqrt{15^m}}{4^m} \omega_m$, to obtain
  \begin{displaymath}
    E \ge \left(\frac{\sqrt{15^m}}{4^m} \omega_m  (\varsigma_{m+1} d(x_0))^m \right)^l
    \left( \frac{3\eta_0^{m+1}}{4(m+1)! d(x_0)} \right)^p 
    = C(m,l,p) d(x_0)^{ml - p} \,.
  \end{displaymath}
  Therefore
  \begin{equation}
    \label{def:R0}
    \tfrac 12 d(\Sigma) 
    = \tfrac 12 \lim_{\varepsilon \to 0^+} (1 + \varepsilon) d(\Sigma) 
    \ge \tfrac 12 d(x_0) \ge C(m,l,p) E^{\frac{-1}{\lambda}} = R_0 \,. \qedhere
  \end{equation}
\end{proof}

\subsection{Removing the dependence on $M_{\theta\beta}$ and $R_{\theta\beta}$}
In this section we show that if $\Sigma$ is $m$-fine with finite
$\E_p^l$-energy, then the constants $M_{\theta\beta}$ and $R_{\theta\beta}$ from
Theorem~\ref{thm:regularity} can be chosen depending solely on $E$, $m$, $l$
and~$p$.

\begin{prop}
  \label{prop:abs-M-theta}
  Let $\Sigma \subseteq \R^n$ be an $m$-fine set such that $\E_p^l(\Sigma) \le E
  < \infty$ for some $p > ml$. Then there exists $R_1 = R_1(E,m,l,p)$ such that
  $\Sigma$ satisfies \eqref{def:ahlfors} and \eqref{def:theta-beta} with
  constants $M_{\theta\beta} = 5$, $R_{\theta\beta} = R_{\ahl} = R_1$ and
  $A_{\ahl} = \frac{\sqrt{15^m}}{4^m}\omega_m$.
\end{prop}

\begin{proof}
  From Theorem~\ref{thm:regularity} and Theorem~\ref{thm:uahlreg} we already
  know that $\Sigma$ is $(\frac 14,m)$-admissible with $\Sigma^* = \Sigma$ and
  satisfies~\eqref{def:ahlfors} with $R_{\ahl} = R_0$ and $A_{\ahl} =
  \frac{\sqrt{15^m}}{4^m}\omega_m$. Hence, by Proposition~\ref{prop:beta-est},
  we also have
  \begin{displaymath}
    \forall r \le R_0 \ 
    \forall x \in \Sigma
    \quad
    \nbeta(x,r) \le C(m,l,p) E^{\frac 1\kappa} r^{\frac{\lambda}{\kappa}} \,.
  \end{displaymath}
  Fix a~point $x_0 \in \Sigma$ and a~radius $r \le R_0$. Choose some $m$-plane
  $P \in G(n,m)$ such that
  \begin{equation}
    \label{eq:P-def}
    \forall y \in \Sigma \cap \CBall(x_0,r)\ \ 
    |\pperp_P(y - x_0)| \le \nbeta(x_0,r) \,.
  \end{equation}
  For brevity we set $\beta = 2\nbeta(x_0,r)$ and $\gamma =
  \frac{\sqrt{15}}{4}$. Inspecting the proof
  of~Proposition~\ref{prop:big-proj-fat-simp} we can find $i \in \N$ such that
  $\rho_i \le r < \rho_{i+1}$. We set $H = H_i$.  Let $y \in \R^n$ be any point
  such that $y - x_0 \in H$ and $|y - x_0| = \gamma r$. We see that $\Sphere(y,
  \tfrac 14 r) \cap (y + H^{\perp})$ is linked with $\Sigma$, hence
  (cf. Proposition~\ref{prop:intpoint}) there exists $z \in \Sigma \cap \Ball(y,
  \tfrac 14 r) \cap (y + H^{\perp})$. Note that $\gamma r \le |z - x_0| \le r$,
  so
  \begin{displaymath}
    \frac{|\pperp_P(z - x_0)|}{|z - x_0|} 
    \le \frac{\beta r}{\gamma r}
    = \frac{\beta}{\gamma} \,,
    \quad \text{hence} \quad
    (z - x_0) \in \Cone \left(
      \big( 1 - \tfrac{\beta^2}{\gamma^2} \big)^{\frac 12}, P^{\perp}
    \right) \cap \Cone(\gamma,H^{\perp}) \,.
  \end{displaymath}
  To apply Proposition~\ref{prop:two-cones} we need to ensure the condition
  \begin{equation}
    \label{cond:R-Sigma}
    \sqrt{1 - \gamma^2} + \tfrac{\beta}{\gamma} 
    \le (1 - \tfrac{\beta}{\gamma})\sqrt{1 - \left(\tfrac{\beta}{\gamma}\right)^2} 
    \iff 
    \beta 
    \le \gamma \left(
      (1 - \tfrac{\beta}{\gamma})\sqrt{1 - \left(\tfrac{\beta}{\gamma}\right)^2} - \sqrt{1 - \gamma^2}
    \right) \,.
  \end{equation}
  Substituting $\Psi = \frac{\beta}{\gamma}$ in~\eqref{cond:R-Sigma} and
  recalling that $\gamma = \frac{\sqrt{15}}4$ we obtain the following inequality
  \begin{equation}
    \label{eq:Psi}
    \Psi \le (1 - \Psi)\sqrt{1 - \Psi^2} - \frac 14 \,.
  \end{equation}
  Note that if $\Psi \to 0$ then the right-hand side converges to $\frac
  34$. Let $\Psi_0$ be the smallest, positive root of~the equation $\Psi = (1 -
  \Psi)\sqrt{1 - \Psi^2} - \frac 14$. Then any $\Psi \in (0,\Psi_0)$ satisfies
  \eqref{eq:Psi}. Recall that $\tfrac 12 \beta = \nbeta(x,r) \le C(m,l,p)
  E^{1/\kappa} r^{\lambda/\kappa}$, so to ensure condition \eqref{cond:R-Sigma}
  it suffices to impose the following constraint
  \begin{equation}
    \label{def:R1}
    r \le 
    \min \left\{
      \left( \frac{\gamma \Psi_0}{C(m,l,p)} \right)^{\frac{\kappa}{\lambda}} E^{\frac{-1}{\lambda}},
      R_0
    \right\}
    = R_1(E,m,l,p) \,.
  \end{equation}
  Now, for such $r$ we can use Proposition~\ref{prop:two-cones} to obtain
  \begin{displaymath}
    H^{\perp} \subseteq \Cone\big(\tfrac 14, H\big) \cap \Cone\big(\tfrac{\beta}{\gamma}, P\big) \,.
  \end{displaymath}

  \begin{figure}[!htb]
    \centering
    \includegraphics{ahlreg9.mps}
    \caption{If $\beta$ is small enough, then the cone
      $\Cone(\tfrac{8\beta}{7\gamma}, P)$ contains $H^{\perp}$ and we can
      continuously transform $S_1$ into $S_3$ inside the conical cap
      $\Cone(\tfrac{8\beta}{7\gamma}, P, \frac 78r\gamma, \frac 78r)$.}
    \label{F:adm-fine}
  \end{figure}
  We set $\mathbf{C} = x_0 + \Cone\big(\tfrac 14, H, \tfrac 12 \rho_i,
  \rho_{i+1}\big)$ and $S_1 = \Sphere(x_0,r) \cap (x_0 + H^{\perp}) \subseteq
  \mathbf{C}$.  Observe that $\mathbf{C} \cap \Sigma = \emptyset$ and
  $\oplk(S_1,\Sigma) = 1$. Using Corollary~\ref{cor:sph-in-cone} we rotate $S_1$
  into $S_2 = \Sphere(x_0,r) \cap (x_0 + P^{\perp})$ (see
  Figure~\ref{F:adm-fine}) inside $\Cone(\tfrac{\beta}{\gamma}, P, r\gamma,
  r)$. Note that for $x \in \Sigma$ such that $|x - x_0| > \gamma r$ we have
  \begin{displaymath}
    \frac{\pperp_P(x - x_0)}{|x - x_0|} 
    < \frac{\beta r}{\gamma r} = \frac{\beta}{\gamma}\,,
  \end{displaymath}
  hence the conical cap $\Cone\big(\tfrac{\beta}{\gamma}, P, \gamma r, r\big)$
  does not intersect $\Sigma$ and the resulting sphere $S_2$ is still linked
  with $\Sigma$. Next we decrease the radius of~$S_2$ to the value $\beta r$
  obtaining another sphere $S_3 = \Sphere(x_0,\beta r) \cap (x_0 + P^{\perp})$
  which is also linked with $\Sigma$.

  We can translate $S_3$ along any vector $v \in P$ with $|v| \le \sqrt{1 -
    \beta^2}r$ without changing the linking number. This way we see that for any
  point $w \in (x_0 + P) \cap \CBall(x_0,\sqrt{1 - \beta^2}r)$ there
  exists a~point $z \in \Sigma$ such that $|z-w| \le \beta r$.

  For any other point $w \in (x_0 + P)$ with $\sqrt{1 - \beta^2}r \le |w - x_0|
  \le r$ we set 
  \begin{displaymath}
    \tilde{w} = w - (w - x_0)|w - x_0|^{-1} (1 - \sqrt{1 - \beta^2})r \,,
  \end{displaymath}
  so that $|\tilde{w} - x_0| \le \sqrt{1 - \beta^2}r$. Then we find $z \in
  \Sigma$ such that $|\tilde{w} - z| \le \beta r$ and we obtain the estimate
  \begin{align*}
    |z - w| &\le |z - \tilde{w}| + |\tilde{w} - w|
    \le \beta r + (1 - \sqrt{1 - \beta^2}) r \\
    &= r \left( \beta + \frac{\beta^2}{1 + \sqrt{1 - \beta^2}} \right)
    \le 2 \beta r = 4 \nbeta(x,r) r \,.
  \end{align*}
  This implies that $\HD(\Sigma \cap \CBall(x_0,r), (x_0 + P) \cap
  \CBall(x_0,r)) \le 5 \nbeta(x_0,r)$. Therefore the infimum over all $H \in
  G(n,m)$ must be even smaller, so $\ntheta(x_0,r) \le 5 \nbeta(x_0,r)$ for any
  $r \le R_{\theta\beta} = R_1$ and we can safely set $M_{\theta\beta} = 5$.
\end{proof}



\section{Uniform estimates on the local graph representations}
\label{sec:tangent-planes}

For the sake of~brevity we introduce the following notation
\begin{displaymath}
  \proj_x = \proj_{T_x\Sigma}
  \qquad \text{and} \qquad
  \pperp_x = \pperp_{T_x\Sigma} \,,
\end{displaymath}
where $x \in \Sigma$. The main result of~this section is
\begin{thm}
  \label{thm:C1tau}
  Let $\Sigma \subseteq \R^n$ be an $m$-fine set. If $\E_p^l(\Sigma) \le E <
  \infty$ for some $p > ml$, then $\Sigma$ is a~closed
  $C^{1,\lambda/\kappa}$-manifold (by~Theorem~\ref{thm:regularity}) and there
  exist constants $R_{\laka} = R_{\laka}(E,m,l,p)$ and~$C_{\laka} =
  C_{\laka}(E,m,l,p)$ such that for all $x \in \Sigma$ there exists a~function
  $F_x : T_x\Sigma \to (T_x\Sigma)^{\perp}$ of class $C^{1,\lambda/\kappa}$ such
  that
  \begin{displaymath}
    (\Sigma - x) \cap \Ball_{R_{\laka}} = \left\{ (y,F_x(y)) \in \R^n : y \in T_x\Sigma \right\} \cap \Ball_{R_{\laka}}
  \end{displaymath}
  \begin{displaymath}
    \text{and} \quad
    \forall y,z \in T_x\Sigma
    \quad
    \| DF_x(y) - DF_x(z) \| \le C_{\laka} |y-z|^{\frac{\lambda}{\kappa}} \,.
  \end{displaymath}
\end{thm}

To prove this theorem we fix a~point $x \in \Sigma$ and for each radii $r > 0$
we choose an $m$-plane $P(x,r)$. Then we use the fact that $\ntheta(x,r) \le
M_{\theta\beta} \nbeta(x,r)$ together with Proposition~\ref{prop:beta-est} to
show that $P(x,r)$ converge to the tangent plane $T_x\Sigma$, when $r \to
0$. This also gives a~bound on the oscillation of~$T_x\Sigma$. Then we derive
Lemma~\ref{lem:1point-rad}, which says that at~some small scale we cannot have
two distinct points $y$ and $z$ of~$\Sigma$ such that the vector $v = (y - z)$
is orthogonal to $T_x\Sigma$. Any such vector $v$ would be close to the tangent
plane $T_{z}\Sigma$ and this would violate the bound on the oscillation
of~tangent planes proved earlier. From here, it follows that there exists
a~small radius $R_{\laka}$ such that $\Sigma \cap \Ball(x,R_{\laka})$ is a~graph
of~some function $F_x$, which is of~class~$C^{1,\lambda/\kappa}$
by~Theorem~\ref{thm:regularity}.

In the sequel of this section we always assume that $\Sigma$ satisfies
hypothesis of~Theorem~\ref{thm:C1tau}.

\subsection{Estimates on the oscillation of tangent planes}
Combining Propositions~\ref{prop:abs-M-theta} and~\ref{prop:beta-est} we see
that
\begin{equation}
  \label{est:beta-indep}
  \forall r \le R_1 \ 
  \forall x \in \Sigma
  \quad
  \ntheta(x,r) \le 5 \nbeta(x,r) \le 5 C(m,l,p) E^{\frac 1\kappa} r^{\frac{\lambda}{\kappa}} \,.
\end{equation}
Let $\tilde{R}_1 = \tilde{R}_1(E,m,l,p) \in (0,R_1]$ be such that $5 C(m,l,p)
E^{\frac 1\kappa} r^{\frac{\lambda}{\kappa}} \le \frac 14$ for all $r \le
\tilde{R}_1$, so $\tilde{R}_1 = C_0 E^{-1/\lambda}$ for some $C_0 = C_0(m,l,p)$.

\begin{lem}
  \label{lem:bap-osc}
  Choose a~point $x \in \Sigma$ and fix some $r_0 \le \tilde{R}_1$. Choose
  another point $y \in \Sigma \cap \CBall(x,\frac 12 r_0)$ and some $r_1 \in
  \left[ \frac 12 r_0, r_0 - |x-y| \right]$. Let $H_0 \in \BAP(x,r_0)$ and $H_1
  \in \BAP(y,r_1)$. Then there exists a constant $C_{hh} = C_{hh}(m,l,p)$ such
  that
  \begin{displaymath}
    \dgras(H_0,H_1) \le C_{hh} E^{1/\kappa} r_0^{\tau} \,.
  \end{displaymath}
\end{lem}

\begin{figure}[!htb]
  \centering
  \includegraphics{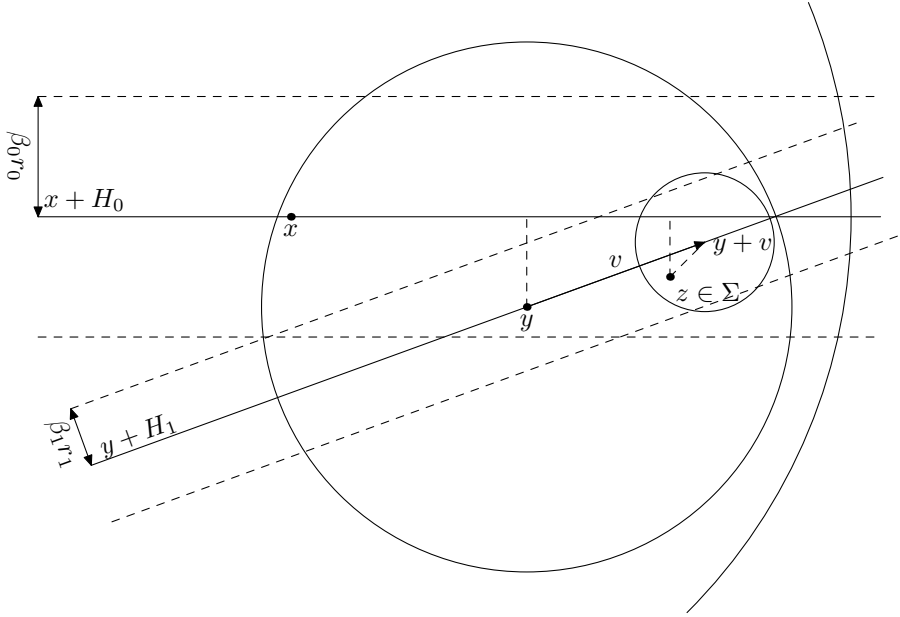}
  \caption{The existence of~$z \in \Sigma$ is guaranteed by the
    condition~\eqref{def:theta-beta}. This allows us to estimate
    $\dgras(H_0,H_1)$.}
  \label{F:bap-osc}
\end{figure}

\begin{proof}
  Set $\beta_0 = \nbeta(x,r_0)$ and $\beta_1 = \nbeta(y,r_1)$. Note that $r_1
  \le \tilde{R}_1$, so $5\beta_1 \le \frac 14$. Let $v \in H_1$ be any vector
  of~length $|v| = r_1 (1-5\beta_1)$. Since $\ntheta(y,r_1) \le 5\beta_1$, there
  exists a~point $z \in \Sigma \cap \CBall(y+v, 5\beta_1 r_1)$. Hence $|(y+v) -
  z| \le 5\beta_1 r_1$ (see Figure~\ref{F:bap-osc}). Note that $\CBall(y+v,
  5\beta_1 r_1) \subseteq \CBall(y,r_1) \subseteq \CBall(x,r_0)$. Therefore
  $\dist(z,x + H_0) = |\pperp_{H_0}(z - x)| \le \beta_0 r_0$ and we obtain the
  estimate
  \begin{align*}
    |\pperp_{H_0}(v)| &\le |\pperp_{H_0}((y-x)+v)| + |\pperp_{H_0}(y-x)| \\
    &\le |((y-x)+v) - (z-x)| + |\pperp_{H_0}(z-x)| + |\pperp_{H_0}(y-x)| \\
    &\le 5 \beta_1 r_1 + \beta_0 r_0 + \beta_0 r_0 \le 7 C E^{1/\kappa} r_0^{1+\lambda/\kappa} \,.
  \end{align*}
  Since $v$ was chosen arbitrarily we get the following estimate for any
  unit vector $e \in H_1 \cap \Sphere$
  \begin{displaymath}
    |\pperp_{H_0}(e)| \le 7 C E^{1/\kappa} \frac{r_0^{1+\lambda/\kappa}}{r_1 (1-5\beta_1)} 
    \le 7 C E^{1/\kappa} \frac{4r_0^{1+\lambda/\kappa}}{3r_1} \,.
  \end{displaymath}
  Recall that $r_1 \ge \frac 12 r_0$, so we have $|\pperp_{H_0}(e)| \le
  \tfrac{8 \cdot 7}{3} C E^{1/\kappa} r_0^{\lambda/\kappa}$. Applying
  Proposition~\ref{prop:red-ang} we get
  \begin{displaymath}
    \dgras(H_0,H_1) \le \tilde{C}(m,l,p) C_{\red}(m) E^{1/\kappa} r_0^{\lambda/\kappa} \,. \qedhere
  \end{displaymath}
\end{proof}

\begin{lem}
  \label{lem:tan-conv}
  Choose a~point $x \in \Sigma$. For each $r \le \tilde{R}_1$ fix an $m$-plane
  $P(r) \in \BAP(x,r)$. There exists a~limit $\lim_{r \to 0} P(r) = T_x\Sigma
  \in G(n,m)$ and it does not depend on the choice of~$P(r) \in \BAP(x,r)$.
\end{lem}

\begin{proof}
  Set $\rho_k = 2^{-k}\tilde{R}_1$ and for each $k$ choose $P_k \in
  \BAP(x,\rho_k)$. Set $\beta_k = \nbeta(x,\rho_k)$. We will show that $\{ P(r)
  \}_{r < \tilde{R}_1}$ satisfies the Cauchy condition. Fix some $0 < s < t <
  \rho_0$ and find two natural numbers $a < b$ such that $\rho_{b+1} < s \le
  \rho_b$ and $\rho_{a+1} < t \le \rho_a$.

  Applying Lemma~\ref{lem:bap-osc} with $x = y$, $r_0 = \rho_j$ and $r_1 =
  \frac 12 r_0 = \rho_{j+1}$ we obtain
  \begin{displaymath}
    \dgras(P_j,P_{j+1}) \le CE^{1/\kappa} \rho_j^{\lambda/\kappa} \,.
  \end{displaymath}
  Setting $r_0 = \rho_b$ and $r_1 = s$ or $r_0 = \rho_a$ and $r_1 = t$
  we also get
  \begin{displaymath}
    \dgras(P(s),P_b) \le CE^{1/\kappa} \rho_b^{\lambda/\kappa}
    \quad \text{and} \quad
    \dgras(P(t),P_a) \le CE^{1/\kappa} \rho_a^{\lambda/\kappa} \,.
  \end{displaymath}
  Using these estimates we can write
  \begin{multline*}
    \dgras(P(r),P(s)) 
    \le \dgras(P(r),P_a) + \sum_{j=a}^{b-1} \dgras(P_j,P_{j+1}) + \dgras(P_b,P(s)) \\
    \le CE^{1/\kappa} \left( \rho_a^{\lambda/\kappa} + \sum_{j=a}^b \rho_j^{\lambda/\kappa} \right) 
    = C E^{1/\kappa} \rho_a^{\lambda/\kappa} \left(1 + \sum_{j=0}^{b-a} 2^{-j \lambda/\kappa} \right) 
    = \hat{C}(m,l,p) E^{1/\kappa} \rho_a^{\lambda/\kappa} \,,
  \end{multline*}
  which shows that the Cauchy condition is satisfied, so $P(r)$ converges
  in~$G(n,m)$ to some $m$-plane, which must be the tangent plane $T_x\Sigma$.
\end{proof}

\begin{cor}
  \label{cor:tan-dist}
  There exists a~constant $C_{th} = C_{th}(m,l,p)$ such that for all $x \in
  \Sigma$, all $r \le \tilde{R}_1$ and all $H \in \BAP(x,r)$ we have
  \begin{displaymath}
    \dgras(T_x\Sigma, H) \le C_{th} E^{1/\kappa} r^{\lambda/\kappa}
  \end{displaymath}
\end{cor}

\begin{cor}
  \label{cor:tan-point}
  There exists a~constant $C_{tp} = C_{tp}(m,l,p)$ such that for all $x \in
  \Sigma$ and all $y \in \Sigma \cap \CBall(x,\tilde{R}_1)$ we have
  \begin{displaymath}
    \dist(y, x+T_x\Sigma) = |\pperp_x(y-x)| \le C_{tp} E^{1/\kappa} |y-x|^{1+\lambda/\kappa} \,.
  \end{displaymath}
\end{cor}

\begin{proof}
  Choose an $m$-plane $H \in \BAP(x,|y-x|)$. Using~\eqref{est:beta-indep} and
  Corollary~\ref{cor:tan-dist} we get
  \begin{align*}
    |\pperp_x(y-x)| 
    &\le |\pperp_H(y-x)| + |\pperp_x (\proj_H(y-x))| \\
    &\le |y-x| \nbeta(x,|y-x|) + |y-x| C_{th} E^{1/\kappa} |y-x|^{\lambda/\kappa} \\
    &\le C_{tp} E^{1/\kappa} |y-x|^{1+\lambda/\kappa} \,. \qedhere
  \end{align*}
\end{proof}

\begin{lem}
  \label{lem:tan-osc}
  There exists a~constant $C_{tt} = C_{tt}(m,l,p)$ such that for all $x
  \in \Sigma$ and for all $y \in \Sigma \cap \CBall(x,\tfrac 12 \tilde{R}_1)$ we
  have
  \begin{displaymath}
    \dgras(T_x\Sigma, T_y\Sigma) \le C_{tt} E^{1/\kappa} |x-y|^{\lambda/\kappa} \,.
  \end{displaymath}
\end{lem}

\begin{proof}
  Let $y \in \Sigma \cap \CBall(x,\frac 12 \tilde{R}_1)$. Set $r_0 = 2 |x-y|$ and
  $r_1 = |x-y|$. Choose any $H_0 \in \BAP(x,r_0)$ and any $H_1 \in
  \BAP(y,r_1)$. From Lemma~\ref{lem:bap-osc} we have
  \begin{displaymath}
    \dgras(H_0,H_1) \le C E^{1/\kappa} r_0^{\lambda/\kappa} \,.
  \end{displaymath}
  On the other hand Corollary~\ref{cor:tan-dist} says that
  \begin{displaymath}
    \dgras(T_x\Sigma, H_0) \le C_{th} E^{1/\kappa} r_0^{\lambda/\kappa}
    \qquad \text{and} \qquad
    \dgras(T_y\Sigma, H_1) \le C_{th} E^{1/\kappa} r_0^{\lambda/\kappa} \,.
  \end{displaymath}
  Putting these estimates together we obtain
  \begin{displaymath}
    \dgras(T_x\Sigma, T_y\Sigma) 
    \le \dgras(T_x\Sigma,H_0) + \dgras(H_0,H_1) + \dgras(H_1,T_y\Sigma)
    = \bar{C} E^{1/\kappa} |x-y|^{\lambda/\kappa} \,. \qedhere
  \end{displaymath}
\end{proof}

\subsection{Uniform estimates on the size of maps}

Combining Corollary~\ref{cor:tan-point} and Lemma~\ref{lem:tan-osc} one can see
that if we have two distinct points $y,z \in \Sigma$ such that $y-z \perp
T_x\Sigma$ and $|y-z| \lesssim |x-y|$ then the tangent plane $T_y\Sigma$ must
form a~large angle with the plane $T_x\Sigma$. Such situation can only happen
far away from $x$ because of~the bound on the oscillation of~tangent planes.

\begin{rem}
  \label{rem:R-ang-bound}
  Let $\iota = \iota(m) = \frac{\ere}{100}$. Lemma~\ref{lem:tan-osc} allows us
  to find a~radius $\tilde{R}_2 = C(m,l,p) E^{-1/\lambda} \in (0, \tilde{R}_1]$
  such that whenever $|x-y| \le \tilde{R}_2$ for some $x,y \in \Sigma$, then
  $\dgras(T_x\Sigma, T_y\Sigma) \le \iota$.
\end{rem}

\begin{lem}
  \label{lem:1point-rad}
  Choose any point $x \in \Sigma$. There exists a~radius $R_2 = C(m,l,p)
  E^{-1/\lambda} \in (0,\tilde{R}_2]$ such that if $y,z \in \Sigma \cap
  \CBall(x,\frac 12 R_2)$ and $(y-z) \perp T_x\Sigma$, then necessarily
  $\max\{|x-y|,|x-z|\} > R_2$.
\end{lem}

\begin{proof}
  Let $C_{tp}$ be the constant from Corollary~\ref{cor:tan-point}. Choose two
  points $y,z \in \Sigma$ such that $(z-y) \perp T_x\Sigma$ and
  $\max\{|x-y|,|x-z|\} \le \frac 12 \tilde{R}_1 (C_0 C_{tp})^{-1}$. Without loss
  of~generality we can assume that $|x-z| \le |x-y| \le 1$, hence
  \begin{displaymath}
    |x-z|^{1+\lambda/\kappa} 
    \le |x-y|^{1+\lambda/\kappa} 
    \le \tilde{R}_1^{\lambda/\kappa}|x-y| 
    \le C_0 E^{-1/\kappa} |x-y| \,.
  \end{displaymath}
  First we estimate the distance $|y-z|$ using Corollary~\ref{cor:tan-point}.
  \begin{align}
    \label{est:yz}
    |y-z| &= |\pperp_x(y-z)|
    \le |\pperp_x(y-x)| + |\pperp_x(x-z)| \\
    &\le C_{tp} E^{1/\kappa} ( |y-x|^{1 + \lambda/\kappa} + |x-z|^{1 + \lambda/\kappa})
    \le 2 C_{tp} C_0 |x-y| \le \tilde{R}_1 \notag \,.
  \end{align}
  Hence we can use Corollary~\ref{cor:tan-point} once again to estimate the
  distance between $z$ and~$T_y\Sigma$. Using the definition of~$\dgras$ we may
  write
  \begin{align}
    \label{est:TxTy-lower}
    \dgras(T_x\Sigma, T_y\Sigma)
    &\ge |z-y|^{-1} |\proj_x(z-y) - \proj_y (z-y)|
    = |z-y|^{-1} |\proj_y (z-y)| \\
    &\ge |z-y|^{-1} \left( |z-y| - |\pperp_y(z-y)| \right) \notag \\ 
    &\ge |z-y|^{-1} \left( |z-y| - C_{tp} E^{1/\kappa} |z-y|^{1+\lambda/\kappa} \right) \notag \\
    &= 1 - C_{tp} E^{1/\kappa} |z-y|^{\lambda/\kappa} \notag \,.
  \end{align}
  On the other hand Lemma~\ref{lem:tan-osc} gives us
  \begin{equation}
    \label{est:TxTy-upper}
    \dgras(T_x\Sigma, T_y\Sigma) \le C_{tt} E^{1/\kappa} |x-y|^{\lambda/\kappa} \,.
  \end{equation}
  Putting these two estimates together we have
  \begin{align*}
    && 1 - C_{tp} E^{1/\kappa} |z-y|^{\lambda/\kappa} 
    &\le \dgras(T_x\Sigma, T_y\Sigma) 
    \le C_{tt} E^{1/\kappa} |x-y|^{\lambda/\kappa} \,, \\
    &\text{so by~\eqref{est:yz}}&
    1 - C_{tp} E^{1/\kappa} (2C_{tp} C_0 |x-y|^{\lambda/\kappa})
    &\le C_{tt}  E^{1/\kappa} |x-y|^{\lambda/\kappa} \,, \\
    &\text{hence}&
    |x-y| &\ge \hat{C}(m,l,p) E^{-1/\lambda} \,.
  \end{align*}
  We set $R_2 = \frac 12 \min\{ \tilde{R}_2 (C_0 C_{tp})^{-1},
  \hat{C}(m,l,p) E^{-1/\lambda} \}$.
\end{proof}

\begin{cor}
  \label{cor:graph}
  For each $x \in \Sigma$ and each $y \in \Sigma \cap \CBall(x,R_2)$ the point
  $y$ is the only point in~the intersection $\Sigma \cap (y + T_x\Sigma^{\perp})
  \cap \CBall(x,R_2)$. Therefore $(\Sigma - x) \cap \CBall_R$ is a~graph of~the
  function
  \begin{align}
    \label{def:Fx}
    F_x : \Dom(x)  &\to T_x\Sigma^{\perp} \cap \CBall_{R_2}
    \quad \text{defined by}\\
    F_x(w) + w &= (\Sigma - x) \cap (w+T_x\Sigma^{\perp}) \cap \CBall_{R_2} \notag \,,
  \end{align}
  where $\Dom(x) = \pi_x(\Sigma \cap \CBall_{R_2}) \subseteq T_x\Sigma$. By
  Theorem~\ref{thm:regularity} the function $F_x$ is of class
  $C^{1,\lambda/\kappa}$.
\end{cor}

Fix a point $o \in \Sigma$. We define the parameterization
\begin{equation}
  \label{def:phi}
  \varphi : \Dom(o) \to \Sigma \cap \CBall(o,R_2) 
  \quad \text{by} \quad
  \varphi(x) = o + F_o(x) + x \,.
\end{equation}
Recall our convention, that when we write $T_o\Sigma$ we always mean the
appropriate subspace of~$\R^n$. For $x \in \Dom(o)$ we set
\begin{displaymath}
  L_x = \left( \proj_o|_{T_{\varphi(x)}\Sigma} \right)^{-1} : T_o\Sigma \to T_{\varphi(x)}\Sigma 
  \quad \text{and} \quad
  K_x = \left( \pperp_o|_{T_{\varphi(x)}\Sigma^{\perp}} \right)^{-1} : T_o\Sigma^{\perp} \to T_{\varphi(x)}\Sigma^{\perp} \,.
\end{displaymath}
Observe that these mappings are well defined since $R_2$ is not greater than
$\tilde{R}_2$ defined in Corollary~\ref{rem:R-ang-bound}, which ensures that
$\dgras(T_o\Sigma, T_{\varphi(x)}\Sigma) \le \iota$. Observe that for any unit
vector $v \in T_{\varphi(x)}\Sigma$ we have $|Q_o v| = | \pi_o v -
\pi_{\varphi(x)}v| \le \iota$, hence $|\pi_o v| = |v - Q_o v| \ge 1 - \iota$.
This shows that the norms $\|L_x\|_{T_o\Sigma}$ and
$\|K_x\|_{T_o\Sigma^{\perp}}$ are less or~equal to~$(1 - \iota)^{-1}$.

\begin{rem}
  \label{rem:deriv}
  Recall that $\iota < \frac 12$. For $x \in \Dom(o)$ and $h \in T_o\Sigma$ we
  have (cf.~\cite[Lemma~3.15]{slawek-phd})
  \begin{align*}
    DF_o(x) h = L_x h - h = Q_o(L_x h)
    \quad &\text{and} \quad
    D\varphi(x) h = L_x h \,, \\
    \text{hence} \quad
    \| DF_o(x) \| \le \frac{\iota}{1 - \iota} < 1
    \quad &\text{and} \quad 
    \| D\varphi(x) \| \le \frac{1}{1 - \iota} < 2 \,.
  \end{align*}
\end{rem}

\begin{rem}
  \label{rem:domain}
  For all $x \in \Dom(o)$ we have $\| D \varphi(x) \| < 2$ and in consequence
  $|\varphi(x) - \varphi(o)| < 2 |x-o|$. Hence $T_o\Sigma \cap \CBall_{\frac
    12 R_2} \subseteq \Dom(o)$.
\end{rem}

\begin{lem}
  \label{lem:tp-deriv}
  Let $C_{\red}$ be the constant from Proposition~\ref{prop:red-ang}. For any
  $x,y \in \Dom(o)$ we have
  \begin{align*}
    \| D\varphi(x) - D\varphi(y) \|
    &\le 4 \dgras(T_{\varphi(x)}\Sigma, T_{\varphi(y)}\Sigma) \\
    \text{and} \quad
    \dgras(T_{\varphi(x)}\Sigma, T_{\varphi(y)}\Sigma)
    &\le C_{\red} \| D\varphi(x) - D\varphi(y) \| \,.
  \end{align*}
\end{lem}

\begin{proof}
  We want to estimate
  \begin{displaymath}
    \| D\varphi(x) - D\varphi(y) \| = \| DF_o(x) - DF_o(y) \| = \| L_x - L_y \| \,.
  \end{displaymath}
  Let $h \in \Sphere$ and set $u = L_x(h)$ and $v = L_y(h)$. Note that $u-v
  \in T_o\Sigma^{\perp}$ so we can write
  \begin{align*}
    |L_x(h) - L_y(h)| 
    &= |u-v| 
    = |K_x(\pperp_x(u-v))|
    \le 2 |\pperp_x(u-v)|
    = 2|\pperp_x(v)| \\
    &\le 2 |v| \dgras(T_{\varphi(x)}\Sigma,T_{\varphi(y)}\Sigma) 
    \le 4 \dgras(T_{\varphi(x)}\Sigma,T_{\varphi(y)}\Sigma) \,.
  \end{align*}

  To prove the second part of Lemma~\ref{lem:tp-deriv} we will use
  Proposition~\ref{prop:red-ang}. Let $(e_1,\ldots,e_m)$ be some orthonormal
  basis of~$T_o\Sigma$. For each $i = 1,\ldots,m$ set $u_i = D\varphi(x)e_i$ and
  $v_i = D\varphi(y)e_i$. Then $(u_1,\ldots,u_m)$ is a~basis
  of~$T_{\varphi(x)}\Sigma$ and $(v_1,\ldots,v_m)$ is a~basis
  of~$T_{\varphi(y)}\Sigma$. By Remark~\ref{rem:deriv} for $i,j = 1,\ldots,m$
  and $i \ne j$ we have
  \begin{displaymath}
    1 \le |u_i| \le \frac{1}{1 - \iota} < 2 \iota 
    \quad \text{and} \quad
    |\langle u_i, u_j \rangle|
    = |\langle DF_o(x)e_i + e_i, DF_o(x)e_j + e_j \rangle| 
    < 3 \iota \,.
  \end{displaymath}
  These estimates show that $(u_1,\ldots,u_m)$ is a~$\red$-basis
  of~$T_{\varphi(x)}\Sigma$ with $\rho = 1$ and $\varepsilon = 3\iota$.
  Moreover
  \begin{displaymath}
    |u_i - v_i| = |D\varphi(x)e_i - D\varphi(y)e_i| 
    \le \| D\varphi(x) - D\varphi(y) \| \,.
  \end{displaymath}
  Since $3\iota = \frac{3\ere}{100} \le \ere$ we can use
  Proposition~\ref{prop:red-ang} to obtain
  \begin{displaymath}
    \dgras(T_{\varphi(x)}\Sigma, T_{\varphi(y)}\Sigma)
    \le C_{\red} \| D\varphi(x) - D\varphi(y) \| \,.
    \qedhere
  \end{displaymath}
\end{proof}

\begin{proof}[Proof of~Theorem~\ref{thm:C1tau}]
  Combining Lemma~\ref{lem:tan-osc} with Lemma~\ref{lem:tp-deriv} we get
  \begin{displaymath}
    \| DF_o(x) - DF_o(y) \| 
    = \| D\varphi(x) - D\varphi(y) \| 
    \le 4 C_{tt} E^{1/\kappa} |x-y|^{\lambda/\kappa}
  \end{displaymath}
  for all $x,y \in \Dom(o) = \pi_o(\Sigma \cap \CBall_{R_2}) \subseteq
  T_o\Sigma$. Since $\pi_o$ is continuous and $\Sigma \cap \CBall_{R_2}$ is
  compact, the function $F_o : \Dom(o) \to T_o\Sigma^{\perp}$ can be extended to
  a~function $F_o : T_o\Sigma \to T_o\Sigma^{\perp}$ without increasing its
  H{\"o}lder norm and in such a way that $\{ (y,F_o(y)) : y \in T_o\Sigma
  \setminus \Dom(o) \} \cap \CBall_{R_2} = \emptyset$. Hence we may set
  \begin{displaymath}
    R_{\laka} = R_2 = C(m,l,p)E^{-\frac{1}{\lambda}}
    \quad \text{and} \quad
    C_{\laka} = 4 C_{tt} E^{\frac {1}{\kappa}} \,. \qedhere
  \end{displaymath}
\end{proof}



\section{Optimal H\"older regularity}
\label{sec:improved-holder}

In~the previous paragraph we showed that $\Sigma$ is a~closed manifold of~class
$C^{1,\lambda/\kappa}$ but $\lambda/\kappa$ was not an optimal exponent. Now we
shall prove that for any $o \in \Sigma$ the map $F_o$ is of~class
$C^{1,\alpha}$, where $\alpha = 1 - \tfrac{ml}{p}$. For this purpose we employ
a~technique developed by Strzelecki, Szuma{\'n}ska and von~der~Mosel
in~\cite{MR2668877}.

The key to the proof of Theorem~\ref{thm:optimal-reg} is
Lemma~\ref{lem:bootstrap}. It says that the oscillation of~$D\varphi$ on~a~ball
of~radius~$r$ can be bounded above by the oscillation of~$D\varphi$ on a ball
of~radius $r/N$, where $N$ is some big number, plus a term of order
$r^{\alpha}$. If we choose $N$ big enough, then, upon iteration, the first term
disappears and the sum of the second terms is still of order $r^{\alpha}$.

To prove Lemma~\ref{lem:bootstrap} we choose two points $x,y \in \Dom(o)$ and we
set $r = |x-y|$. From Lemma~\ref{lem:tp-deriv} we know that the oscillation
of~$D\varphi$ is comparable with the oscillation of $T_{\varphi(\cdot)}\Sigma$.
We choose points $x_0,\ldots,x_m$ and $y_0,\ldots,y_m$ near $x$ and $y$
respectively, such that $\{x_i - x_0\}_{i=1}^m$ and $\{y_i - y_0\}_{i=1}^m$ form
a~roughly (up to an error of order $\frac 1k$, where $k$ is some big number)
orthogonal bases of $T_o\Sigma$. Moreover $|x_i - x_0| \approx r/N$ and $|y_i -
y_0| \approx r/N$. In the scale we are working in, we always have $\|D\varphi\|
\le 1 + \iota$, so $\{\varphi(x_i) - \varphi(x_0)\}_{i=1}^m$ and $\{\varphi(y_i)
- \varphi(y_0)\}_{i=1}^m$ are also roughly (up to an error of order $\frac 1k +
\iota$) orthogonal and span some $m$-dimensional secant spaces $X$ and $Y$
respectively. If we choose the points $y_0,\ldots,y_m$ appropriately, then the
``angle'' $\dgras(X,Y)$ can be estimated by~$r^{\alpha}$. The error we make when
we pass from $\dgras(T_{\varphi(x}, T_{\varphi(y)})$ to $\dgras(X,Y)$ is
comparable with the oscillation of~$D\varphi$ on balls of radius~$r/N$.

To choose ``good'' points $y_0,\ldots,y_m$ we first define the set of ``bad
parameters'' $\bad(x_0,\ldots,x_{l-2})$, i.e. such $z \in \Dom(o)$ that the
integrand
\begin{displaymath}
  \Kphi(x_0,\ldots,x_{l-2},z) 
  = \sup_{p_l,\ldots,p_{m+1} \in \Sigma} \K(\varphi(x_0),\ldots,\varphi(x_{l-2}),z,p_l,\ldots,p_{m+1})
\end{displaymath}
is big. From finiteness of~$\E_p^l(\Sigma)$, we derive the conclusion that the
measure of~$\bad(x_0,\ldots,x_{l-2})$ has to be smaller than the measure of a
ball of~radius $r/(kN)$, hence close to each $\tilde{y}$ there exists $y$ which
does not belong to $\bad(x_0,\ldots,x_{l-2})$. From the fact that
$\Kphi(x_0,\ldots,x_{l-2},y)$ is small, we derive an~estimate on the distance
of~$\varphi(y)$ from $\varphi(x_0) + X$, which in turn gives
the~estimate~$\dgras(X,Y) \lesssim r^{\alpha}$.

In the sequel of this section we always assume that $\Sigma$ satisfies
hypothesis of~Theorem~\ref{thm:C1tau}, $o \in \Sigma$ is fixed, $\varphi$
is~given by~\eqref{def:phi} and $l$ is a fixed number from the set
$\{1,2,\ldots,m+2\}$.

\subsection{Bootstrapping the H{\"o}lder exponent}

Let $S \subseteq \Dom(o)$ be any set and $r \le \frac 12 R_{\laka}$. We define
the oscillation of $\varphi$ on $S$ as follows
\begin{displaymath}
  \Phi(r,S) = \sup \big\{
  \| D\varphi(x) - D\varphi(y) \| :
  x,y \in S,\, |x-y| \le r
  \big\} \,.
\end{displaymath}
For $x,y \in T_o\Sigma$ we set
\begin{displaymath}
  \Disc_r = T_o\Sigma \cap \Ball_r \,,
  \quad
  \Disc(x,r) = x + \Disc_r
  \quad \text{and} \quad
  \Dxy(x,y) = \Disc_{|x-y|} + \tfrac{x+y}{2} \subseteq T_o\Sigma \,.
\end{displaymath}
and we define
\begin{displaymath}
  M_p^l(a,\rho) = \left( \E_p^l\big(\varphi(\CDisc(a,\rho))\big) \right)^{\frac 1p}
  \quad \text{and} \quad
  E_p^l(x,y) = \E_p^l\big(\varphi(\CDxy(x,y))\big)  \,.
\end{displaymath}
Note that if we set $|J\varphi(x)| = \sqrt{\det((D\varphi(x))^t D\varphi(x))}$ and
\begin{displaymath}
  \Kphi(x_0,\ldots,x_{l-1}) = \sup_{p_l,\ldots,p_{m+1} \in \Sigma} \K(\varphi(x_0),\ldots,\varphi(x_{l-1}),p_l,\ldots,p_{m+1}) \,,
\end{displaymath}
\begin{equation}
  \label{eq:Epl}
  \text{then}\quad
  E_p^l(x,y) = \int_{[\CDxy(x,y)]^l}  \Kphi(x_0,\ldots,x_{l-1})^p\ |J \varphi(x_0)| \cdots |J \varphi(x_{l-1})|\ dx_0 \cdots\ dx_{l-1} \,.
\end{equation}

\begin{lem}
  \label{lem:bootstrap}
  For all $k \ge k_0 = 100/\ere$ and $N \ge N_0 = 8$ there exist constants $C_1
  = C_1(m)$ and $C_2 = C_2(m,l,p,k,N)$ such that for all $x,y \in \Disc_{\frac
    16 R_{\laka}}$
  \begin{equation}
    \label{est:Dphi-osc}
    \| D\varphi(x) - D\varphi(y) \| 
    \le C_1 \Phi\left(\tfrac{2|x-y|}{N}, \CDxy(x,y)\right) + C_2 E(x,y)^{\frac 1p} |x-y|^{\alpha} \,.
  \end{equation}
\end{lem}

Using this lemma we can prove Theorem~\ref{thm:optimal-reg}.

\begin{proof}[Proof of Theorem~\ref{thm:optimal-reg}]
  Fix some $a~\in \CDisc_{\frac 1{12} R_{\laka}}$ and a~radius $R \in
  (0,\frac{1}{36} R_{\laka}]$. Taking the supremum on both sides
  of~\eqref{est:Dphi-osc} over all $x,y \in \CDisc(a,R)$ satisfying $|x-y| \le r
  \le R$ we obtain the estimate
  \begin{displaymath}
    \Phi(r, \CDisc(a,R))
    \le C_1 \Phi \left( \tfrac{2r}N, \CDisc(a,R+r) \right)
    + C_2 M_p^l(a,R+r) r^{\alpha} \,.
  \end{displaymath}
  Choose any $j \in \N$. Iterating the above inequality $j$ times we get
  \begin{displaymath}
    \Phi(r, \CDisc(a,R))
    \le C_1^j \Phi \left( 2^j N^{-j} r, \CDisc(R + r_j) \right) 
    + C_2 M_p(a,R + r_j) r^{\alpha} 
    \sum_{l=0}^{j-1} \left( \frac{C_1}{N^{\alpha}} \right)^l \,,
  \end{displaymath}
  where $r_j = r \sum_{l=0}^{j-1} 2^l N^{-l} \le 2r$. Recall that we know
  a~priori that $\varphi$ is $C^{1,\lambda/\kappa}$-smooth, so we can estimate
  the~first term on the right-hand side by
  \begin{displaymath}
    \Phi \big( 2^j N^{-j} r, \CDisc(a,R + r_j) \big)
    \le C_{\laka} 2^{j \lambda/\kappa} N^{-j \lambda/\kappa} r^{\lambda/\kappa} \,,
  \end{displaymath}
  \begin{displaymath}
    \text{which gives} \quad
    \Phi(r, \CDisc(a,R))
    \le C_{\laka} (C_1 N^{-\lambda/\kappa})^j r^{\lambda/\kappa}
    + C_2 M_p^l(a,3R) r^{\alpha} \sum_{l=0}^{j-1} (C_1 N^{-\alpha})^l 
  \end{displaymath}
  for each $j \in \N$. To ensure that the first term disappears and that
  the~second term converges when $j \to \infty$ we need to know the following
  \begin{equation}
    \label{cond:N2}
    C_1 2^{\lambda/\kappa} N^{-\lambda/\kappa} < 1
    \quad \text{and} \quad
    C_1 N^{-\alpha} < 1 \,.
  \end{equation}
  Since $C_1 = C_1(m)$, we can find $N = N(m,l,p) \ge N_0$ for which
  condition~\eqref{cond:N2} is satisfied. Passing with $j$ to the limit $j \to
  \infty$ we obtain the bound
  \begin{displaymath}
    \Phi(r, \CDisc(a,R))
    \le C_2 M_p^l(a,3R) \sum_{l=0}^{\infty} (C_1 N^{-\alpha})^l r^{\alpha}
    = C(m,l,p) M_p^l(a,3R) r^{\alpha} \,.
  \end{displaymath}
  Hence, for any $x,y \in \CDisc_{\frac{1}{36}R_{\laka}}$, taking $a =
  \frac{x+y}2$ and $r = R = |x-y|$ we get
  \begin{displaymath}
    \|D\varphi(x) - D\varphi(y)\| \le C(m,l,p) M_p^l\big(\tfrac{x+y}{2},3|x-y|\big) |x-y|^{\alpha} \,. \qedhere
  \end{displaymath}
\end{proof}

\begin{proof}[Proof of Lemma~\ref{lem:bootstrap}]
  Let us fix $x,y \in \Disc_{\frac 16 R_{\laka}}$. Since $|x-y| < \frac 13
  R_{\laka}$ and $\frac{|x+y|}2 < \frac 16 R_{\laka}$, we~have $\Dxy(x,y)
  \subseteq \Disc_{\frac 12 R_{\laka}}$. Let $x_0,\ldots,x_{l-2} \in
  \Dxy(x,y)$. We define the sets of~\emph{bad parameters}
  \begin{displaymath}
    \bad(x_0, \ldots, x_{l-2}) = \left\{
      z \in \CDxy(x,y) :
      \Kphi(x_0, \ldots, x_{l-2},z)^p \ge \frac{(kN)^m}{|x-y|^{ml} \omega_m^l} E_p^l(x,y)
    \right\}
  \end{displaymath}
  Recalling~\eqref{eq:Epl} and using the fact that $|J\varphi| \ge 1$ we can
  estimate the measure of $\bad(x_0, \ldots, x_{l-2})$ as follows
  \begin{align}
    E_p^l(x,y)
    &\ge \int_{[\Dxy(x,y)]^{l-1}} \int_{\bad(x_0,\ldots,x_{l-2})} \Kphi(x_0,\ldots,x_{l-2},z)\ dz\ dx_0 \cdots dx_{l-2} \notag \\
    &\ge \omega_m^{l-1} |x-y|^{m(l-1)} \HM^m(\bad(x_0,\ldots,x_{l-2})) \frac{(kN)^m}{|x-y|^{ml} \omega_m^l} E_p^l(x,y) \notag \\
    \label{est:bad-meas}
    &\iff \quad
    \HM^m(\bad(x_0,\ldots,x_{l-2})) \le \omega_m \left( \frac{|x-y|}{kN} \right)^m \,.
  \end{align}

  Fix an orthonormal basis $(e_1,\ldots,e_m)$ of~$T_o\Sigma$. For $i =
  1,\ldots,m$ we set
  \begin{align*}
    x_0 &= x \,,& x_i &= x_0 + \tfrac{|x-y|}{N}e_i \,,&
    \tilde{y}_0 &= y &\text{and}&& \tilde{y}_i &= \tilde{y}_0 + \tfrac{|x-y|}{N}e_i \,.
  \end{align*}
  Estimate~\eqref{est:bad-meas} shows that we can find
  \begin{displaymath}
    y_0,\ldots,y_m \in \CDxy(x,y) \setminus \bad(x_0,\ldots,x_{l-2}) \,,
    \quad \text{such that} \quad
    |y_i - \tilde{y}_i| \le \frac{|x-y|}{kN}
  \end{displaymath}
  for each $i = 0, \ldots, m$. We set
  \begin{displaymath}
    X = \opspan\{ \varphi(x_i) - \varphi(x_0) \}_{i=1}^m 
    \quad \text{and} \quad
    Y = \opspan\{ \varphi(y_i) - \varphi(y_0) \}_{i=1}^m \,.
  \end{displaymath}
  Using Lemma~\ref{lem:tp-deriv} we obtain
  \begin{multline}
    \label{est:Dphi-xy}
    \| D\varphi(x) - D\varphi(y) \|
    \le \| D\varphi(x) - D\varphi(x_0) \| 
    + \| D\varphi(x_0) - D\varphi(y_0) \| 
    + \| D\varphi(y_0) - D\varphi(y) \| \\
    \le 2\Phi\left(\tfrac{|x-y|}{kN},\CDxy(x,y)\right) 
    + 4 \dgras(T_{\varphi(x_0)}\Sigma,T_{\varphi(y_0)}\Sigma) \\
    \le 2\Phi\left(\tfrac{|x-y|}{kN},\CDxy(x,y)\right) 
    + 4\dgras(T_{\varphi(x_0)}\Sigma,X) 
    + 4\dgras(X,Y) 
    + 4\dgras(Y,T_{\varphi(y_0)}\Sigma) \,.
  \end{multline}

  For each $i = 1,\ldots,m$, from the fundamental theorem of calculus we have
  \begin{align*}
    v_i &= \varphi(x_i) - \varphi(x_0) = \int_0^1 \tfrac{d}{dt} \left(
      \varphi(x_0 + t(x_i - x_0))
    \right)\ dt \\
    &= \int_0^1 \left( D\varphi(x_0 + t(x_i - x_0)) - D\varphi(x_0) \right) (x_i - x_0) \ dt
    + D\varphi(x_0) (x_i - x_0) \\
    &= \sigma_i + w_i \,.
  \end{align*}
  Observe that $w_1,\ldots,w_m$ forms a basis of $T_{\varphi(x_0)}\Sigma$ and
  $v_1,\ldots,v_m$ forms a basis of $X$. Using the above estimate we see that
  \begin{displaymath}
    |v_i - w_i| = |\sigma_i| 
    \le \Phi\big(|x_i - x_0|, \Dxy(x,y)\big) |x_i - x_0|
    = \Phi\Big( \tfrac{|x-y|}N, \Dxy(x,y)\Big) \tfrac{|x-y|}N  \,,
  \end{displaymath}
  Let $a_i = x_i - x_0 = \frac{|x-y|}N e_i$ and $b_i = F_o(x_i) - F_o(x_0)$. Then $v_i
  = a_i + b_i$. From Remark~\ref{rem:deriv} we know that $|b_i| \le 2\iota |a_i|
  = \frac{|x-y|}{50 N} \ere$, hence
  \begin{displaymath}
    \frac{|x-y|^2}{N^2} \left(\delta_i^j - \frac{\ere}{25} - \frac{\ere^2}{50^2}\right)
    \le |\langle v_i, v_j \rangle| = |\langle a_i + b_i, a_j + b_j \rangle| \le
    \frac{|x-y|^2}{N^2} \left(\delta_i^j + \frac{\ere}{25} + \frac{\ere^2}{50^2}\right) \,.
  \end{displaymath}
  Applying Proposition~\ref{prop:red-ang} we come to
  \begin{equation}
    \label{est:Tx}
    \dgras(T_{\varphi(x_0)}\Sigma, X) 
    \le C_{\red} \Phi\Big(\tfrac{|x-y|}N,\Dxy(x,y)\Big) \,.
  \end{equation}

  We estimate $\dgras(T_{\varphi(y_0)}\Sigma, Y)$ in a~similar way. For $i =
  1,\ldots,m$ we define $\bar{v}_i$, $\bar{w}_i$, $\bar{a}_i$ and~$\bar{b}_i$ as
  follows
  \begin{displaymath}
    \bar{a}_i = y_i - y_0 \,, \qquad \bar{b}_i = F_o(y_i) - F_o(y_0) \,,
  \end{displaymath}
  \begin{displaymath}
    \bar{v}_i = \varphi(y_i) - \varphi(y_0) = \bar{a}_i + \bar{b}_i 
    \quad \text{and} \quad
    \bar{w}_i = D\varphi(y_0) (y_i - y_0) \,,
  \end{displaymath}
  so that $Y = \opspan\{\bar{v}_1,\ldots,\bar{v}_m\}$ and
  $T_{\varphi(y_0)}\Sigma = \opspan\{\bar{w}_1,\ldots,\bar{w}_m\}$. Again, using
  the fundamental theorem of calculus, we get
  \begin{displaymath}
    |\bar{v}_i - \bar{w}_i| =
    \le \Phi\big(|y_i - y_0|, \Dxy(x,y)\big) |y_i - y_0|
    \le 2 \Phi\Big( \tfrac{2|x-y|}N, \Dxy(x,y)\Big) \tfrac{|x-y|}N  \,,
  \end{displaymath}
  Recall that $k \ge 100/\ere$. It is easy to verify that
  \begin{displaymath}
    \frac{|x-y|^2}{N^2} \left( \delta_i^j - \frac 8k \right)
    \le |\langle \bar{a}_i, \bar{a}_j \rangle| \le 
    \frac{|x-y|^2}{N^2} \left( \delta_i^j + \frac 8k \right) \,,
  \end{displaymath}
  which implies that $|\bar{b}_i| \le 2\iota |\bar{a}_i| \le \frac{|x-y|}{25 N}
  \ere$. Therefore
  \begin{displaymath}
    \frac{|x-y|^2}{N^2} \left(\delta_i^j - \ere \right)
    \le |\langle \bar{v}_i, \bar{v}_j \rangle| = |\langle \bar{a}_i + \bar{b}_i, \bar{a}_j + \bar{b}_j \rangle| \le
    \frac{|x-y|^2}{N^2} \left(\delta_i^j + \ere\right) 
  \end{displaymath}
  and we can apply Proposition~\ref{prop:red-ang} once more obtaining
  \begin{equation}
    \label{est:Ty}
    \dgras(T_{\varphi(y_0)}\Sigma, Y) 
    \le 2 C_{\red} \Phi\Big(\tfrac{2|x-y|}N,\Dxy(x,y)\Big) \,.
  \end{equation}
  Combining estimates \eqref{est:Ty}, \eqref{est:Tx} and \eqref{est:Dphi-xy} and
  using Lemma~\ref{lem:tp-deriv} we get
  \begin{equation}
    \label{est:Dphi-xy2}
    \| D\varphi(x) - D\varphi(y) \|
    \le C_1(m) \Phi \left(\tfrac{2|x-y|}{kN},\CDxy(x,y)\right) + 4 \dgras(X,Y) \,.
  \end{equation}
  Hence, we only need to estimate $\dgras(X,Y)$.

  Observe that for each $z \in \Dxy(x,y) \setminus \bad(x_0,\ldots,x_{l-2})$ we have
  \begin{displaymath}
    \Kphi(x_0,\ldots,x_{l-2},z) \le
    \frac{(kN)^{m/p}}{\omega_m^{l/p}\,|x-y|^{ml/p}} E_p^l(x,y)^{1/p} \,.
  \end{displaymath}
  Directly from the definition of $\Kphi$ we also have
  \begin{multline*}
    \Kphi(x_0,\ldots,x_{l-2},z) 
    \ge \K(\varphi(x_0),\ldots,\varphi(x_{m}),\varphi(z))  \\
    = \frac{\HM^m(\simp(\varphi(x_0),\ldots,\varphi(x_m)) \dist(\varphi(z), \varphi(x_0) + X))}
    {(m+1) \diam(\varphi(x_0),\ldots,\varphi(x_m),\varphi(z))^{m+2}} \\
    \ge \frac{\HM^m(\simp(x_0,\ldots,x_m)) \dist(\varphi(z), \varphi(x_0) + X)}{(m+1) (2|x-y|)^{m+2}} 
    = \frac{\dist(\varphi(z), \varphi(x_0) + X)}{(m+1)! N^m 2^{m+2} |x-y|^{2}} \,.
  \end{multline*}
  Hence
  \begin{displaymath}
    \dist(\varphi(z), \varphi(x_0) + X) 
    \le C(m,l,p,k,N) E_p^l(x,y)^{1/p} |x-y|^{1 - \frac{ml}{p}} \frac{|x-y|}{N} \,.
  \end{displaymath}
  We have shown already that $\bar{v}_1,\ldots,\bar{v}_m$ forms a~$\red$-basis
  of~$Y$ with $\rho = \frac{|x-y|}{N}$ and $\varepsilon = \ere$. Moreover, since
  $y_i \notin \bad(x_0,\ldots,x_{l-2})$, we have
  \begin{align*}
    \dist(\bar{v}_i,X) = |\pperp_{X} \bar{v}_i| 
    &\le \dist(\varphi(y_i), \varphi(x_0) + X) + \dist(\varphi(y_0), \varphi(x_0) + X) \\
    &\le 2 C(m,l,p,k,N) E_p^l(x,y)^{1/p} |x-y|^{1 - \frac{ml}{p}} \frac{|x-y|}{N} \,.
  \end{align*}
  Thence, by Proposition~\ref{prop:red-ang}, the following holds
  \begin{displaymath}
    \dgras(X,Y) \le \tilde{C}(m,l,p,k,N) E_p^l(x,y)^{1/p} |x-y|^{1 - \frac{ml}{p}} \,.
  \end{displaymath}
  Together with~\eqref{est:Dphi-xy2} this gives~\eqref{est:Dphi-osc} and
  Lemma~\ref{lem:bootstrap} is proven.
\end{proof}




\section*{Acknowledgements}
Much of~the material presented in~this article was contained in~the author's
doctoral thesis written at~University of~Warsaw under the direction of~Pawe{\l}
Strzelecki. I~wish~to thank Professor Strzelecki for all his helpful advice
and~encouragement.



\bibliography{menger}{}
\bibliographystyle{hplain}

\end{document}